\documentclass{article}
\usepackage[english]{babel}
\usepackage{a4wide}
\usepackage[T1]{fontenc}
\usepackage{amsmath}
\usepackage{amssymb}
\usepackage{amsthm}
\usepackage{graphicx}
\usepackage{mathrsfs}

\def\N{\mathbb{N}}

\def\R{\mathbb{R}}
\def\pa{\partial}
\def\var{\varepsilon}
\def\epsilon{\varepsilon}
\def\ov{\overline}

\newtheorem{theorem}{Theorem}
\newtheorem{propo}{Proposition}
\newtheorem{lem}{Lemma}
\newtheorem{corol}{Corollary}

\newtheorem{remark}{Remark}
\newtheorem{Theo}{Theorem}[section]
\newtheorem{definition}[Theo]{Definition}
\def\footnote{\@ifnextchar[{\@xfootnote}{\stepcounter {\@mpfn}\xdef\@thefnmark{\thempfn}\@footnotemark\@footnotetext}}
\newcommand{\ud}{\mathrm{d}}

\begin{document}

\begin{center}
\Large{New results for triangular reaction cross diffusion system}
 \end{center}
\bigskip

\medskip
\centerline{\scshape L. Desvillettes \& A. Trescases}
\medskip
{\footnotesize
  \centerline{CMLA, ENS Cachan, CNRS}
  \centerline{61 Av. du Pdt. Wilson, F-94230 Cachan, France}
\centerline{E-mails: desville@cmla.ens-cachan.fr, trescase@cmla.ens-cachan.fr}}
\bigskip

\begin{abstract}
We present an approach based on entropy and duality methods for ``triangular'' reaction cross diffusion systems of two
equations, in which cross diffusion terms appear only in one of the equations. Thanks to this approach, we recover and extend
many existing results on the classical ``triangular'' Shigesada-Kawasaki-Teramoto model.
\end{abstract}

\noindent{Subject Class: } 35K57 35B25 35Q92 92D25
\medskip


\section{\protect\bigskip Introduction}

Reaction cross diffusion equations naturally appear in physics (cf. \cite{BGS} for example) as well as in population dynamics. 
We are interested here in the study of a class of systems first introduced by Shigesada, Kawasaki, and Teramoto (cf. \cite{STK}). Those systems aim
at modeling the repulsive effect of populations of two different species in competition, and are possibly leading to the apparition of patterns (cf. \cite{izmi}).
\medskip

The unknowns are the quantities $u:=u(t,x)\ge 0$ and $v:= v(t,x)\ge 0$. They represent the  number densities of the two considered species (say, species 1 and species 2). They depend on the time variable $t\in\R_+$ and the space variable $x\in\Omega$. Hereafter, $\Omega$ is a smooth bounded domain of $\mathbb{R}^N$ ($N\in\N^*:= \N - \{0\}$) and we denote by $n=n(x)$ its unit normal outward vector at point $x\in\pa\Omega$. The original model of \cite{STK} writes
 \begin{equation} \label{skt} \left\{\begin{aligned}
 \pa_t u - \Delta_x (d_u\, u + d_{11}\, u^2 + d_{12} \,u\,v) = u\, (r_u - r_{a}\, u - r_{b}\,v) &\qquad \text{in } \R_+\times\Omega, \\
 \pa_t v - \Delta_x (d_v\, v + d_{21}\, u\,v + d_{22} \,v^2) = v\, (r_v - r_{c}\,v - r_d\, u) &\qquad \text{in } \R_+\times\Omega, \\
 \nabla_x u \cdot n = \nabla_x v\cdot n = 0 &\qquad \text{on } \R_+\times\pa\Omega.
 \end{aligned} \right.
 \end{equation} 

The coefficients $r_u, r_v>0$ are the growth rates in absence of other individuals, $r_a, r_b, r_c, r_d>0$ correspond to the logistic inter- and intraspecific competition effects, and $d_u, d_v>0$ are the diffusion rates. The coefficients $d_{ij} \ge 0$ 
($i,j=1,2$) represent the repulsive effect: individuals of species $i$ increase their diffusion rate in presence of individuals of their own species when $d_{ii}>0$ (self diffusion) or of the other species when $d_{ij}>0$ ($i \neq j$, cross diffusion). 
\medskip

In the sequel, we shall only consider the case when $d_{21} =0$ and $d_{12}>0$, which is sometimes called ``triangular''. In such a situation,
 the second equation is coupled to the first one only through the competition (reaction) term
while the first one is coupled to the second one through both diffusion and competition terms
 (the fully coupled system when $d_{21} >0$ and $d_{12}>0$ has a quite different
mathematical structure, cf. \cite{CJ} and \cite{DLM} for example). We shall also only focus on the case when no self diffusion appears (that is $d_{11}=d_{22}=0$) since this case is the most studied one: note however that the presence of self-diffusion (that is, $d_{11}>0$ and/or $d_{22}>0$) usually helps to obtain better bounds on the solution. As a consequence,  our results  are expected to hold when self-diffusion is present.
\medskip

Under the extra assumptions detailed above, the Shigesada-Kawasaki-Teramoto system writes 
 \begin{equation} \label{sktt} \left\{\begin{aligned}
 \pa_t  u - \Delta_x (d_u\, u + d_{12} \,u\,v) = u\, (r_u - r_{a}\, u - r_{b}\,v) &\qquad \text{in } \R_+\times\Omega, \\ 
 \pa_t  v - d_v\, \Delta_x v  = v\, (r_v - r_{c}\,v - r_d\, u) &\qquad \text{in } \R_+\times\Omega, \\
 \nabla_x u \cdot n = \nabla_x v\cdot n = 0 &\qquad \text{on } \R_+\times\pa\Omega.
 \end{aligned} \right.
 \end{equation} 
\par
Following \cite{imn}, this system can be seen as the formal singular
limit of a reaction diffusion system which writes
\begin{equation} \label{skktinf} \left\{\begin{aligned}
	\partial_{t} u_{A}^\epsilon - d_u  \,\Delta_x u_{A}^\epsilon = [r_u-r_{a}\,(u_{A}^\epsilon+u_{B}^\epsilon)-r_{b}\,v^\epsilon]\, u_{A}^\epsilon + \frac{1}{\epsilon}[k(v^\epsilon)\,u_{B}^\epsilon-h(v^\epsilon)\,u_{A}^\epsilon] &\qquad \text{in } \R_+\times\Omega,\\
	\partial_{t} u_{B}^\epsilon - (d_u+d_B) \,\Delta_x u_{B}^\epsilon = [r_u-r_{a}\,(u_{A}^\epsilon+u_{B}^\epsilon)-r_{b}\,v^\epsilon]\, u_{B}^\epsilon - \frac{1}{\epsilon}[k(v^\epsilon)\,u_{B}^\epsilon-h(v^\epsilon)\,u_{A}^\epsilon] &\qquad \text{in } \R_+\times\Omega,\\
	\partial_{t} v^\epsilon - d_v \,\Delta_x v^\epsilon = [r_v-r_{c}\,v^\epsilon-r_{d}\,
(u_{A}^\epsilon +u_{B}^\epsilon)]\, v^\epsilon &\qquad \text{in } \R_+\times\Omega,\\
 \nabla_x u_A^\epsilon \cdot n =\nabla_x u_B^\epsilon \cdot n = \nabla_x v^\epsilon\cdot n = 0 &\qquad \text{on } \R_+\times\pa\Omega,
\end{aligned} \right.
\end{equation}
where $d_B>0$, and $h,k$ are two (continuous) functions from $\R_+$ to $\R_+$ satisfying (for all $v \ge 0$) the identity
 $$d_B\,\frac{h(v)}{h(v)+k(v)} =  d_{12}\,v. $$
The limit holds (at the formal level)
in the following sense: if $u_A^\var$, $u_B^\var$, and $v^\var$ are solutions to system (\ref{skktinf}) (with 
$\var$-independent initial data), the quantity $(u_A^\var + u_B^\var,v^\var)$
converges towards $(u,v)$, where $u$ and $v$ are solutions to system
(\ref{sktt}). Note that this asymptotics can be biologically meaningful: when $\var>0$, the system (\ref{skktinf}) represents a microscopic model in which the species $u$ can be found in two states (the quiet state $u_A$ and the stressed state $u_B$), and 
the individuals of this species switch
from one state to the other one with a ``large'' rate (proportional to $1/\var$). 
\medskip

We present in this paper results for the existence, uniqueness and stability of a large class of systems including 
(\ref{sktt}). More precisely, we relax the assumption stating that the competition terms are logistic (quadratic), and replace it with the assumption stating that the competition terms are given by power laws (the powers being suitably chosen). We 
also relax the assumption stating that the cross diffusion term is quadratic (that is, proportional to $u\, v$) and replace it by the more
general assumption stating that it writes $u\,\phi(v)$ (with $\phi \in C^1(\R_+)$,  and $\phi$ nonnegative). 
\medskip

Hence, we shall consider the system
\begin{align}
\label{sku1} \pa_t u - \Delta_x (d_u\, u + u\,\phi(v)) = u\, (r_u - r_{a}\, u^a - r_{b}\,v^b) \qquad \text{in } \R_+\times\Omega,\\
\label{sku2} \pa_t v - d_v\, \Delta_x v  = v\, (r_v - r_{c}\,v^c - r_d\, u^d) \qquad \text{in } \R_+\times\Omega,
\end{align}
with homogeneous Neumann boundary conditions
\begin{equation} \label{sku3}
 \nabla_x u \cdot n = \nabla_x v\cdot n = 0 \qquad \text{on } \R_+\times\pa\Omega,
 \end{equation}
and initial data
\begin{equation} \label{sku4}
u(0,\cdot)=u_{in}, \qquad v(0,\cdot)=v_{in} \qquad \text{in } \Omega.
\end{equation}
The functions $u_{in} :=u_{in}(x)\ge0$ and $v_{in} := v_{in}(x)\ge0$ are defined on $\Omega$ and assumed to be nonnegative.
In cases in which we want to prove that the solutions are strong, they will sometimes be required to satisfy the following compatibility conditions on the boundary
\begin{align}
\label{compat_u} \nabla_x u_{in} \cdot n = 0 \qquad \text{on } \pa\Omega,\\
\label{compat_v} \nabla_x v_{in} \cdot n = 0 \qquad \text{on } \pa\Omega.
\end{align}
In our theorems, we shall consider parameters in \eqref{sku1}-\eqref{sku2} which satisfy the

\medskip
\noindent
{\bf{Assumption A}}: $d_u, d_v>0$, $r_u,r_v,r_a,r_b,r_c,r_d>0$, $a,b,c,d>0$, and  $\phi:= \phi(v) \ge 0$, $\phi \in C^1(\R_+)$.
\medskip

We now specify what is meant by a weak solution in our theorems.

\medskip
\noindent
We recall the following notation: for $p\in[1,\infty[$,  $$L^{p}_{\text{loc}}(\R_+\times\ov{\Omega}) :=\{u=u(t,x): \text{for all }T>0,\, \, \int_0^T\int_{\Omega} |u(t,x)|^p\, dx dt < \infty\}.$$

\begin{definition}\label{defimer}
Let $\Omega$ be a smooth bounded domain of $\mathbb{R}^N$ ($N\in \N^\ast$).
Let $u_{in},\,v_{in}$ be two nonnegative functions lying in $L^1(\Omega)$, and  
$d_u, d_v$, $r_u,r_v,r_a,r_b,r_c,r_d$, $a,b,c,d>0$, $\phi:= \phi(v)$ be parameters satisfying assumption A.
\par
 A pair of functions $(u,v)$ such that $u:=u(t,x)\ge0$ and $v:=v(t,x)\ge0$, lying moreover in $L_{\text{loc}}^{\max(1+a,d)}(\R_+\times\ov{\Omega})\times L_{\text{loc}}^{\infty}(\R_+\times\ov{\Omega})$ is a {\bf weak solution} of \eqref{sku1}-\eqref{sku4} if $\nabla_x u$, $\nabla_x v$, $\nabla_x\, [\phi(v)\,u]$ lie in $L_{\text{loc}}^{1}(\R_+\times\ov{\Omega})$ and, for all test functions $\psi_1$, $\psi_2 \in C^1_c(\R_+ \times {\ov{\Omega}})$, 
 the following identities hold:
\begin{equation*}
 - \int_0^{\infty}\int_{\Omega} (\pa_t \psi_1)\, u - \int_{\Omega} \psi_1(0,\cdot)\, u_{in} 
+ \int_0^{\infty}\int_{\Omega} \nabla_x \psi_1 \cdot \nabla_x \left[(d_u+ \phi(v))\,u\right] =  \int_0^{\infty}\int_{\Omega} \psi_1\, u\, (r_u - r_{a}\, u^a - r_{b}\,v^b), 
\end{equation*}
\begin{equation*}
 - \int_0^{\infty}\int_{\Omega} (\pa_t \psi_2)\, v - \int_{\Omega} \psi_2(0,\cdot)\, v_{in} 
+ d_v\int_0^{\infty}\int_{\Omega} \nabla_x \psi_2 \cdot  \nabla_x v = \int_0^{\infty}\int_{\Omega} \psi_2\,
 v\, (r_v - r_{c}\,v^c - r_d\, u^d). 
 \end{equation*}
 Note that all terms in the previous identities are well-defined under our assumptions on 
 $u_{in}$, $v_{in}$, $u$, $v$, $\psi_1$, $\psi_2$, $\phi$.
\end{definition}

We propose two theorems, corresponding to the respective cases $d<a$ and $a\le d$.
The first one writes:
\begin{theorem}\label{theo_ex1}
Let $\Omega$ be a smooth bounded domain of $\mathbb{R}^N$ ($N\in \N^\ast$). 
We suppose that Assumption A on the coefficients of system (\ref{sku1}) -- (\ref{sku2}) holds,
 together with the extra assumption $d<a$. Finally, we consider initial data $u_{in}\ge 0$, $v_{in}\ge 0$, such that
 $u_{in} \in L^{p_0}(\Omega)$, 
 $v_{in}\in L^\infty(\Omega)\cap W^{2, 1 + p_0/d}(\Omega)$ for some $p_0>1$. If $1 + p_0/d \ge 3$, we also assume
 the compatibility condition \eqref{compat_v}.
 \medskip
 
Then, there exists a (global, with nonnegative components) weak solution $(u,v)$ of system (\ref{sku1}) -- (\ref{sku4}) in the sense of Definition~\ref{defimer} [In particular, $(u,v) \in L_{\text{loc}}^{\max(1+a,d)}(\R_+\times\ov{\Omega})\times L_{\text{loc}}^{\infty}(\R_+\times\ov{\Omega})$ 
and $\nabla_x u$, $\nabla_x v$, $\nabla_x\, [\phi(v)\,u]$ lie in $L_{\text{loc}}^{1}(\R_+\times\ov{\Omega})$].
\par
 Moreover, this solution lies in $L^{p_0 + a}_{\text{loc}}(\R_+\times\ov{\Omega})\times L^\infty_{\text{loc}}(\R_+\times\ov{\Omega})$, $\nabla_x v$ lies in $L^{2(1+p_0/d)}_{\text{loc}}(\R_+\times\ov{\Omega})$ and for all $p\in]1,p_0]$, $T>0$,
 \begin{equation}\label{u_Vp0}
\sup_{t\in [0,T]} \int_\Omega u^{p_0} (t) < +\infty \qquad \text{;} \qquad \int_0^T \int_\Omega |\nabla_x (u^{p/2})|^2<+\infty \, .
\end{equation} 
 \medskip
 
We suppose in addition to the previous assumptions that $\phi \in C^2(\R_+)$, $u_{in} \in W^{2,s_0}(\Omega)$, $v_{in} \in  W^{2, \frac{1}{d}\,\max(a s_0,a+2)}(\Omega)$ for some
 $s_0 > 1 + N/2$, and that compatibility conditions \eqref{compat_u}, (resp. \eqref{compat_v}) hold when
 $s_0 \ge 3$, (resp. $ \frac{1}{d}\,\max(a s_0,a+2) \ge 3$). Then $(u,v)$ is H\"older continuous 
on $\R_+ \times \bar{\Omega}$, and $\pa_t u,\, \pa_{x_i x_j} u \in L^{s_0}_{\text{loc}}(\R_+\times\ov{\Omega})$,
$\pa_{x_i} u \in L^{2}_{\text{loc}}(\R_+\times\ov{\Omega})$, $\pa_t v, \, \pa_{x_i x_j} v \in L^{\max(a s_0,a+2)/d}_{\text{loc}}(\R_+\times\ov{\Omega})$ ($i,j = 1..N$, and the derivatives are taken in the sense of distributions). Note that since $u$ is H\"older, we know that $u \in L^{\max(a s_0, a+2)}_{\text{loc}}(\R_+\times\ov{\Omega})$.

 \medskip
 
Finally, if (in addition to the previous assumptions) $\phi$ has H\"older continuous second order derivatives on $\R_+$, if $u_{in}, v_{in}$ have H\"older continuous second order derivatives on $\ov{\Omega}$, and if compatibility conditions \eqref{compat_u}--\eqref{compat_v} are satisfied, then $u,v$ have
 H\"older continuous first order time derivatives and
 H\"older continuous second order space derivatives on
$\R_+ \times \ov{\Omega}$.
\par 
In this last setting, and provided that $b,d\ge 1$,
 the following stability estimate holds:
 if $(u_{1, in}, v_{1,in})$ and
 $(u_{2, in}, v_{2,in})$ are two sets of initial data with nonnegative components, then any corresponding weak solutions $(u_1,v_1)$, $(u_2,v_2)$ in the sense of Definition \ref{defimer}, lying in $ L^{\max(a s_0, a+2)}_{\text{loc}}(\R_+\times\ov{\Omega})\times  L^{\infty}_{\text{loc}}(\R_+\times\ov{\Omega})$ and such that
 (for any $T>0$)
 \begin{equation}\label{ui_V2}
\sup_{t\in [0,T]} \int_\Omega u_i^{2} (t) < +\infty \qquad \text{and} \qquad \int_0^T \int_\Omega |\nabla_x u_i|^2<+\infty \qquad \text{for } i = 1,\,2, 
\end{equation} 
  satisfy (for any $T>0$)
$$||u_1 - u_2||_{L^2([0,T] \times \Omega)} + ||v_1 - v_2||_{L^2([0,T] \times \Omega)}
 \le C_T\, \bigg( ||u_{1,in} - u_{2,in}||_{L^2(\Omega)} 
+ ||v_{1,in} - v_{2,in}||_{L^2(\Omega)} \bigg), $$
 for some constant $C_T>0$. 
As a consequence, uniqueness holds in this last setting (among weak solutions 
in the sense of Definition \ref{defimer} lying in $ L^{\max(a s_0, a+2)}_{\text{loc}}(\R_+\times\ov{\Omega})\times  L^{\infty}_{\text{loc}}(\R_+\times\ov{\Omega})$ and satisfying \eqref{ui_V2}).
\end{theorem}

\begin{remark} The first setting provides global weak solutions. In the second setting, 
those solutions are shown to be strong, in the sense that all derivatives appearing in the equations lie in some $L^p$ with $p\in [1, \infty]$. Finally, in the last setting, those solutions are shown to be classical, in the sense that all derivatives appearing in the equations are continuous. Stability and uniqueness (in the class of weak solutions satisfying some extra regularity) holds when
the assumptions on the parameters imply that weak solutions are classical solutions.
\end{remark}

Then, our second theorem writes
\medskip

\begin{theorem}\label{theo_ex2}
Let $\Omega$ be a smooth bounded domain of $\mathbb{R}^N$ ($N\in \N^\ast$). 
We suppose that Assumption A on the coefficients of system (\ref{sku1}) -- (\ref{sku2}) holds.
We moreover suppose that $a \le d$, $a\le 1$, $d \le 2$.
 Finally, we consider initial data $u_{in}\ge 0$, $v_{in}\ge 0$ such that
 $u_{in} \in L^{2}(\Omega)$, 
 $v_{in}\in L^\infty(\Omega)\cap W^{2, 1 + 2/d}(\Omega)$. If $1 + 2/d \ge 3$ (i.-e. $d\le 1$), we also assume
 the compatibility condition \eqref{compat_v}.
\medskip

Then, there exists a (global, with nonnegative components) weak solution $(u,v)$ of system (\ref{sku1}) -- (\ref{sku4}) in the sense of Definition~\ref{defimer} [In particular, $(u,v) \in L_{\text{loc}}^{\max(1+a,d)}(\R_+\times\ov{\Omega})\times L_{\text{loc}}^{\infty}(\R_+\times\ov{\Omega})$ 
and $\nabla_x u$, $\nabla_x v$, $\nabla_x\, [\phi(v)\,u]$ lie in $L_{\text{loc}}^{1}(\R_+\times\ov{\Omega})$].
\par
 Moreover, $(u,v)$ lies in $L^{2}_{\text{loc}}(\R_+\times\ov{\Omega})\times L^\infty_{\text{loc}}(\R_+\times\ov{\Omega})$,
 $\nabla_x v \in L^{2 + \eta}_{\text{loc}}(\R_+\times\ov{\Omega})$ for some $\eta>0$,
 $u$ satisfies (for all $T>0$, and for some $p>0$)
 \begin{equation}\label{u_Vp}
\sup_{t\in [0,T]} \int_\Omega u (t) < +\infty \qquad \text{;} \qquad \int_0^T \int_\Omega |\nabla_x (u^{p/2})|^2<+\infty \, .
\end{equation} 
\end{theorem}

Those existence theorems are consequences of propositions showing the convergence in a singular perturbation problem.
This problem is analogous to system (\ref{skktinf}) in the case of the Shigesada-Kawasaki-Teramoto model. It writes: 

\begin{equation} \label{eq:3rd} \left\{\begin{aligned}
	\partial_{t} & u_{A}^\epsilon - d_A \,\Delta_x u_{A}^\epsilon = [r_u-r_{a}\,(u_{A}^\epsilon+u_{B}^\epsilon)^a-r_{b}\,(v^\epsilon)^b]\,u_{A}^\epsilon + \frac{1}{\epsilon}\,[k(v^\epsilon)\,u_{B}^\epsilon-h(v^\epsilon)\,u_{A}^\epsilon] \qquad \text{in } \R_+\times\Omega,\\
	\partial_{t} & u_{B}^\epsilon - (d_A+d_B)\, \Delta_x u_{B}^\epsilon = [r_u-r_{a}\,(u_{A}^\epsilon+u_{B}^\epsilon)^a-r_{b}\,(v^\epsilon)^b]\,u_{B}^\epsilon - \frac{1}{\epsilon}\,[k(v^\epsilon)\,u_{B}^\epsilon-h(v^\epsilon)\,u_{A}^\epsilon] \qquad \text{in } \R_+\times\Omega,\\
	\partial_{t} & v^\epsilon - d_v \,\Delta_x v^\epsilon = [r_v-r_{c}\,(v^\epsilon)^c-r_{d}\,(u_{A}^\epsilon +u_{B}^\epsilon)^d]\,v^\epsilon \qquad \text{in } \R_+\times\Omega,\\
\end{aligned} \right.\end{equation}
where $h$ and $k$ lie in $C^1(\R_+)$ and satisfy, for some $h_0>0$,
\begin{equation} \label{adem}
d_A+d_B\, \frac{h(v)}{h(v)+k(v)} = d_u +\phi(v), \qquad h(v) \ge h_0, \qquad k(v)\ge h_0, \qquad \text{for all } v\in \R_+ .
\end{equation}
\par
The existence of $h$ and $k$ in $C^1(\R_+)$ satisfying (\ref{adem}) is a part of the proof of 
Theorems \ref{theo_ex1} and \ref{theo_ex2}.
\par
We add homogeneous Neumann boundary conditions 
\begin{equation}\label{Neumann_3RD}
\nabla_x u_A^\epsilon \cdot n = \nabla_x u_B^\epsilon \cdot n = \nabla_x v^\epsilon \cdot n = 0  \qquad \text{on } \R_+\times\pa\Omega.
\end{equation}
We also add initial data to (\ref{eq:3rd}), (\ref{Neumann_3RD}) thanks to a regularization process that we now describe. Let $(\rho^\var)_{\var>0}$ be a family of mollifiers on $\mathbb{R}^N$,
 and for all $\var>0$, let $\chi^\var$ be a cutoff function (given by Urysohn's lemma) lying in 
$C^{\infty}(\mathbb{R}^N)$, and satisfying
\begin{equation*}
0\le \chi^\var\le 1 \text{ in } \mathbb{R}^N, \qquad \chi^\var=1 \text{ inside } \{x\in\Omega:d(x,\partial\Omega)>2\var\}, \qquad \chi^\var=0 \text{ outside } \{x\in\Omega:d(x,\partial\Omega)>\var\}.
\end{equation*} 
Then, given two nonnegative functions (lying in $L^1(\Omega)$) $u_{in},\, v_{in}$, we define
\begin{equation}
u_{A,in}:=\frac{k(v_{in})}{h(v_{in})+k(v_{in})}\,u_{in}, \qquad u_{B,in}:=\frac{h(v_{in})}{h(v_{in})+k(v_{in})}\,u_{in} \qquad \text{ on } \Omega,
\end{equation}
and extend by zero those functions on $\mathbb{R}^N-\Omega$ (so that the convolution on $\mathbb{R}^N$ can be used). 
\par 
We therefore add to (\ref{eq:3rd}), (\ref{Neumann_3RD}) the regularized initial data 
(defined on $\Omega$):
\begin{equation}\label{initial_3RD}
u_A^\var(0,\cdot) = u_{A,in}^\var := (\chi^\var (u_{A,in}\ast\rho^\var)+\epsilon)|_{\Omega},\quad u_B^\var(0,\cdot) = u_{B,in}^\var := (\chi^\var (u_{A,in}\ast\rho^\var)+\var)|_{\Omega},\quad v^\var (0,\cdot)= v_{in}^\var := v_{in} + \var .
\end{equation}
We shall use in our propositions related to the system (\ref{eq:3rd}), (\ref{Neumann_3RD}),
(\ref{initial_3RD}) the
\medskip

\noindent
{\bf{Assumption B}}: $d_A, d_B, d_u, d_v>0$, $r_u,r_v,r_a,r_b,r_c,r_d>0$, $a,b,c,d>0$. The functions $\phi$, $h$ and $k$ lie in 
$C^1(\R_+)$ and satisfy (\ref{adem}).
\medskip

For the singular perturbation problem with a given $\var \in ]0,1[$, we shall consider strong solutions
defined in the following way:

\begin{definition}\label{defimerd}
Let $\Omega$ be a smooth bounded domain of $\mathbb{R}^N$ ($N\in\N^*$). 
We suppose that Assumption B on the coefficients of system (\ref{eq:3rd}), (\ref{Neumann_3RD}), (\ref{initial_3RD}) holds, and that $u_{in},\,v_{in}$ are
 two nonnegative functions lying in $L^1(\Omega)$. We finally consider $\var\in ]0,1[$.
\smallskip

  A set of nonnegative functions $(u_A^\epsilon,u_B^\epsilon,v^\epsilon)$ such that $u^\epsilon_A:=u^\epsilon_A(t,x)$, $u^\epsilon_B:=u^\epsilon_B(t,x)$ lie in $L_{\text{loc}}^{\max(1+a,d)}(\R_+\times\ov{\Omega})$,
 and $v^\epsilon :=v^\epsilon(t,x)$ lie in $L_{\text{loc}}^{\infty}(\R_+\times\ov{\Omega})$, will be called a {\bf strong solution} of (\ref{eq:3rd}), (\ref{Neumann_3RD}), (\ref{initial_3RD}) if $\pa_t u_A^\epsilon$, $\pa_t u_B^\epsilon$, $\pa_t v^\epsilon$ and $\pa_{x_i,x_j} u_A^\epsilon$, $\pa_{x_i,x_j} u_B^\epsilon$, $\pa_{x_i,x_j} v^\epsilon$ ($i,\,j=1..N$) lie in $L^1_{loc}(\R_+\times\ov{\Omega})$ and equations \eqref{eq:3rd}, \eqref{Neumann_3RD} and \eqref{initial_3RD} are satisfied almost everywhere in $\R_+\times\Omega$ (resp. $\R_+\times\pa\Omega$, $\Omega$).
\end{definition}

Our results concerning the behavior when $\var\to 0$ of the strong
solutions of system \eqref{eq:3rd}, (\ref{Neumann_3RD}), \eqref{initial_3RD} are summarized in the two following propositions (corresponding to the respective cases $d<a$ and $d\ge a$):
\medskip

\begin{propo}\label{theo_sing1}
Let $\Omega$ be a smooth bounded domain of $\mathbb{R}^N$ ($N\in\N^*$). 
We suppose that Assumption B on the coefficients of system (\ref{eq:3rd}), (\ref{Neumann_3RD}), (\ref{initial_3RD}) holds,
 and assume moreover that $d<a$. Finally, we consider initial data $u_{in}\ge 0$, $v_{in}\ge 0$
 such that 
 $u_{in} \in L^{p_0}(\Omega)$,
 $v_{in}\in L^\infty(\Omega)\cap W^{2, 1 + p_0/d}(\Omega)$ for some $p_0>1$. If $1 + p_0/d \ge 3$, we also assume
 the compatibility condition \eqref{compat_v}.
\medskip

Then, for any $\var \in ]0,1[$, there exists a strong (global, with nonnegative components) solution $(u_A^\var,u_B^\var, v^\var)$  in the sense of Definition \ref{defimerd} to system
(\ref{eq:3rd}), (\ref{Neumann_3RD}), (\ref{initial_3RD}).
\medskip

 Moreover, when $\var\to 0$, $(u_A^\epsilon,u_B^\epsilon,v^\epsilon)$ converges, 
up to extraction of a subsequence, for almost every  $(t,x) \in \mathbb{R}_+\times\Omega$ to a limit  
$(u_A,u_B,v)$ lying in $L^{p_0+a}_{\text{loc}}(\R_+\times\ov{\Omega})\times L^{p_0+a}_{\text{loc}}(\R_+\times\ov{\Omega})\times L^\infty_{\text{loc}}(\R_+\times\ov{\Omega})$, and such that $u_A\ge$, $u_B\ge 0$, $v\ge 0$. The $L^{\infty}$ estimate on $v$
can be made explicit:
\begin{equation}\label{vex}
0 \le v(t,x) \le \max \left(||v_{in}||_{L^{\infty}(\Omega)}, \left[\frac{r_v}{r_c\,(c+1)}\right]^{1/c} \right)
\qquad {\hbox{ for a.e. }} \quad (t,x) \in \R_+\times\Omega.
\end{equation}
\par
Furthermore, $\nabla_x v$ lies in $L^{2(1+p_0/d)}_{\text{loc}}(\R_+\times\ov{\Omega})$ and the quantity $u:=u_A+u_B$ satisfies $\nabla_x u, \nabla_x (u\,\phi(v)) \in L^{1}_{\text{loc}}(\R_+\times\ov{\Omega})$, 
and for all $p\in]1,p_0]$, $T>0$,
\begin{equation}\label{es:V2}
\sup_{t\in [0,T]} \int_\Omega u^{p_0} (t) < +\infty \qquad \text{and} \qquad \int_0^T \int_\Omega |\nabla_x (u^{p/2})|^2<+\infty.
\end{equation}
\par
 Finally, $h(v(t,x))\,u_A(t,x)=k(v(t,x))\, u_B(t,x)$ for a.e. 
$(t,x) \in \mathbb{R}_+\times\Omega$, and $(u, v)$ is a 
(global, with nonnegative components) weak solution of system (\ref{sku1}) -- (\ref{sku4}) in the sense of Definition \ref{defimer}. 
\end{propo}
\medskip

\begin{propo}\label{theo_sing2}
Let $\Omega$ be a smooth bounded domain of $\mathbb{R}^N$ ($N\in\N^*$). 
We suppose that Assumption B on the coefficients of system (\ref{eq:3rd}), (\ref{Neumann_3RD}), (\ref{initial_3RD}) holds,
 and assume moreover that
 $a \le d$, $a \le 1$, $d \le 2$.
  Finally, we consider initial data $u_{in}\ge 0$, $v_{in}\ge 0$ such that
 $u_{in} \in L^{2}(\Omega)$,
 $v_{in}\in L^\infty(\Omega)\cap W^{2, 1 + 2/d}(\Omega)$. If $1 + 2/d \ge 3$ (i.-e. $d\le 1$), we also assume
 the compatibility condition \eqref{compat_v}.
\medskip

Then, for any $\var \in ]0,1[$, there exists a strong (global, with nonnegative components) solution $(u_A^\var,u_B^\var, v^\var)$ in the sense of Definition \ref{defimerd} to system
(\ref{eq:3rd}), (\ref{Neumann_3RD}), (\ref{initial_3RD}).
\medskip

 Moreover, when $\var\to 0$, $(u_A^\epsilon,u_B^\epsilon,v^\epsilon)$ converges, 
 up to extraction of a subsequence, for almost every  $(t,x) \in \mathbb{R}_+\times\Omega$ to a limit  
$(u_A,u_B,v)$ lying in $L^{2}_{\text{loc}}(\R_+\times\ov{\Omega})\times L^{2}_{\text{loc}}(\R_+\times\ov{\Omega})\times L^\infty_{\text{loc}}(\R_+\times\ov{\Omega})$,
and such that $u_A\ge 0$, $u_B\ge 0$, $v\ge 0$.
The explicit $L^{\infty}$ estimate on $v$ given by (\ref{vex}) also holds.
Furthermore, 
 $\nabla_x v$ lies in $L^{2 +\eta}_{\text{loc}}(\R_+\times\ov{\Omega})$ for some $\eta>0$,
 and the quantity $u:=u_A+u_B$ satisfies $\nabla_x u, \nabla_x (u\,\phi(v)) \in L^{1}_{\text{loc}}(\R_+\times\ov{\Omega})$
 and for some $p>0$ (and all $T>0$),
\begin{equation}\label{es:Vp}
 \sup_{t\in [0,T]} \int_\Omega u (t) < +\infty \qquad \text{and} \qquad \int_0^T \int_\Omega |\nabla_x (u^{p/2})|^2<+\infty.
\end{equation}
\par
 Finally,  $h(v(t,x))\,u_A(t,x)=k(v(t,x))\, u_B(t,x)$ for a.e. $(t,x) \in \mathbb{R}_+\times\Omega$, and $(u, v)$ is a (global, with nonnegative components) weak solution
 of system (\ref{sku1}) -- (\ref{sku4}) in the sense of Definition \ref{defimer}.
\end{propo}

In the following remarks, we discuss some direct extensions of the results stated above.
\begin{remark}
Theorems \ref{theo_ex1} and \ref{theo_ex2} use classical parabolic ($W^{2,1}_s$ with the notations of \cite{lsu})
 estimates. For the sake of simplicity, we chose to use a non-optimal version of those estimates, formulated below in Proposition
 \ref{theo:parabolic_reg}. Note that the assumptions could be somewhat improved (see \cite{lsu}) 
in Theorems \ref{theo_ex1} and \ref{theo_ex2}: 
first, the estimates do not require a full compatibility condition on the boundary $\pa\Omega$ in the critical case $s=3$; secondly, some of the initial data assumed to belong to  $W^{2,s}(\Omega)$ in our theorems and propositions 
can be assumed to belong only to the fractional Sobolev space $W^{2-2/s,s}(\Omega)$.
\end{remark}

\begin{remark}
In the case of Theorem \ref{theo_ex2}, the compactness of the nonlinear reaction terms $u^{1+a}$ and $u^d$ is obtained thanks
to an $L^p$ estimate for some $p>2$ given by a duality lemma. Notice first that this enables to 
treat coefficients $a=1 + \eta$ and $d= 1+\eta$ 
 when $\eta>0$ is smaller than some (small) constant.
  Secondly, the duality lemma (stated in Lemma \ref{le:duality})
 for initial data in $L^2(\Omega)$ holds in fact for initial data in $L^{2-\eta}(\Omega)$ 
when $\eta>0$ is also smaller than some (small) constant. This allows to replace in
 Theorem \ref{theo_ex2} the assumption $u_{in}\in L^2(\Omega)$ by
 the weaker assumption $u_{in}\in L^{2-\eta}(\Omega)$.
\end{remark}

\begin{remark}
Since (as we shall see later on), $v$ satisfies a maximum principle in Theorems \ref{theo_ex1} and \ref{theo_ex2}, those theorems can easily be extended
in the case when the functions $v \mapsto r_b\,v^b$ and $v\mapsto r_c\,v^c$
are replaced by any smooth functions of $v$ (with an arbitrary growth when $v \to \infty$).
The functions  $u \mapsto r_a\,u^a$ and $u\mapsto r_d\,u^d$ can also be replaced by smooth functions in 
Theorems \ref{theo_ex1} and \ref{theo_ex2}, provided that those functions behave in the same way as $u \mapsto r_a\,u^a$ and $u\mapsto r_d\,u^d$ when
$u \to \infty$.  
\end{remark}

\begin{remark}
In the last setting of Theorem \ref{theo_ex1}, a minimum principle for $v$ allows to replace the assumption stating that $\phi''$ is locally H\"older continuous on $[0,+\infty[$
 by the assumption stating that $\phi''$ is locally H\"older continuous on $]0,+\infty[$, provided that the initial datum for $v$ is bounded below by a strictly positive constant.
\end{remark}
\medskip

The model \eqref{skt} was proposed by Shigesada, Kawasaki and Teramoto in \cite{STK}. For modeling issues, see also \cite{ok}. As far as mathematical analysis is concerned, two directions have been widely investigated in the literature: a series of papers focuses on steady-states and stability (patterns are shown to appear; see \cite{imn} and the references therein); other works concern existence, smoothness and uniqueness of solutions.

The local (in time) existence was established by Amann: in his series of papers \cite{Am1}-\cite{Am3}, he proved a general result of existence of local (in time) solutions for parabolic systems, including \eqref{skt} and \eqref{sku1}-\eqref{sku2}.

The global (in time) existence has then been proved under various assumptions. One of the difficulties which arises  is related to the use of Sobolev inequalities in parabolic estimates, which only provides results in low dimension. Indeed, for the well studied triangular quadratic case (that is, \eqref{skt} with $d_{21}=0$), most papers allowing strong cross diffusion
(that is, when no restriction is imposed on $d_{12}$) only deal 
with low dimensions: for results in dimension 1, see \cite{mami}, \cite{mim} and \cite{shi}. In \cite{yag}, Yagi showed the global existence in dimension 2 in the presence of self diffusion, and Lou, Ni and Wu obtained it in \cite{lnw} without condition on self diffusion, together with a stability result.
Choi, Lui and Yamada first got rid of the restriction on the dimension in \cite{cly1} (without self diffusion in the second equation), provided that the cross diffusion coefficient $d_{12}$ is sufficiently small. In a following paper \cite{cly2}, they removed the smallness assumption on the cross diffusion in the presence of self diffusion in the first equation. However, in the presence of self diffusion in the second equation, they require that the dimension is lower than 6. Finally, Phan improved this result up to dimension lower than 10 in \cite{pha1}, and in any dimension under the assumption that the self diffusion dominates the cross diffusion in \cite{pha2}. For the quadratic system \eqref{sktt} without self diffusion, our Theorem \ref{theo_ex2} gives the existence of global solutions in any dimension, without restriction on the strength of the cross diffusion.

When it comes to systems with general reaction terms of the form \eqref{sku1}-\eqref{sku2}, Posio and Tesei first showed the existence (in any dimension) of global solutions under some strong assumption on the reaction coefficients in \cite{pt}. This assumption was relaxed in \cite{yam} by Yamada, who obtained the existence of global strong solutions under the assumption $a>d$, which is exactly our assumption in Theorem \ref{theo_ex1}. The main differences between our work (in the case $a>d$)
and \cite{yam} are the following: first, our Theorem \ref{theo_ex1} allows singular initial data leading to weak solutions (and provides results very close to those of \cite{yam} when initial data are smooth).
 Then our method, based on simple energy estimates, presents a unifying proof for a wide range of parameters including both the quadratic case and the case $a>d$. Finally, the approximating system that we use leads to self-contained
proofs without reference to abstract existence theorems. 
Note also that (for general reaction terms) Wang got similar results in \cite{wan} in the presence of self diffusion in the first equation, under a condition (depending on the dimension) of smallness of the parameter $d$ w.r.t. the parameter $a$.

Systems of reaction diffusion equations such as \eqref{skktinf} were introduced by Iida, Mimura and Ninomiya in \cite{imn} to approximate cross diffusion systems, in particular from the point of view of stability. The convergence of the stationary problem was explored by Izuhara and Mimura in \cite{izmi}, both numerically and theoretically. In \cite{code}, Conforto and Desvillettes showed the convergence of the solutions of \eqref{skktinf} towards a solution of the system \eqref{sktt} in dimension one. Our paper generalizes their result to a wider set of admissible reaction terms and in any dimension. Note finally that Murakawa obtained similar results for a class of non triangular systems in \cite{mu}.
\bigskip

{\bf Note}: After submission of this article, Hoang, Nguyen and Phan released the paper \cite{hpn}. Therein, they obtain global smooth solutions in any dimension of space for the quadratic case (system \eqref{skt} with $d_{21}=0$) in the presence of self diffusion in the first equation. 
Their result relies on new nonlinear parabolic estimates (that they establish) and uses the regularizing effect of the presence of the self diffusion. 
\par
 The \emph{a priori} estimates obtained thanks to our methods (duality lemma and entropy functional in $L^p$ spaces) 
still hold 
in the case when self-diffusion is present. However, it is not obvious whether or not the singular perturbation method that we use can be extended to this case.

\bigskip

The rest of our paper is structured as follows: Propositions \ref{theo_sing1} and \ref{theo_sing2} are proven in Section \ref{sec2}. Then, 
Section \ref{sec3} is devoted to the proof of Theorems \ref{theo_ex1} and \ref{theo_ex2}.

\section{Proof of the convergence of the singularly perturbed equations}\label{sec2}

We begin with the 
\medskip

\begin{proof}[Proof of  Proposition \ref{theo_sing1}]
We fix $T>0$, and shall write from now on (for any $q\in [1, \infty]$) $L^q := L^q([0,T]\times\Omega)$.
 In the proof of this proposition and of the following proposition, the constant $C_T>0$
 only depends on the parameters $d_A, d_B, d_u, d_v$, $r_u, r_v, r_a, r_b, r_c, r_d$, $a, b, c, d$,
 the domain $\Omega$, the initial data $u_{in}, v_{in}$, the functions $\phi$, $h$ and $k$, and the time $T$. It may also depends on the parameters $p$ and $q$ used later. In this proposition, it also depends on the parameter $p_0$ in the initial datum. 
 In particular, all the estimates are uniform w.r.t $\epsilon \in ]0,1[$, unless stated otherwise.

\medskip

We first observe that for a given $\var\in ]0,1[$, standard theorems for reaction-diffusion equations show the existence of a (global, nonnegative for each component) strong solution $(u_A^\var,u_B^\var, v^\var)$ in the sense of Definition \ref{defimerd} to system
(\ref{eq:3rd}), (\ref{Neumann_3RD}), (\ref{initial_3RD}). Moreover, these solutions satisfy
\begin{equation}\label{es:eps_depend}\begin{split}
\|\pa_t u_A^\var\|_{q}, \, \|\pa_t u_B^\var\|_{q}, \, \|\pa_{x_ix_j} u_A^\var\|_{q}, \, \|\pa_{x_ix_j} u_B^\var\|_{q} \le \mu_{T,\var} \qquad \text{for }i,\,j=1..N, \text{ for all } q>1,\\
\|\pa_t v^\var\|_{1+p_0/d}, \, \|\pa_{x_ix_j} v^\var\|_{1+p_0/d} \le \mu_{T,\var}\quad \text{for }i,\,j=1..N ,\\
\nu^1_{T,\var} \ge u_A^\var(t,x) \qquad \nu^1_{T,\var} \ge u_B^\var(t,x) \qquad \nu^1_{T,\var} \ge v^\var(t,x) \ge \nu^0_{T,\var} >0 \quad \text{a.e. }(t,x) \in [0,T]\times\Omega,
\end{split}\end{equation}
where the constants $\mu_{T,\var}>0$, $\nu^1_{T,\var}>0$, $\nu^0_{T,\var}>0$ depend 
on $\var$ and the other parameters, including $T$, and the last inequality is a direct consequence of the 
minimum principle.
 We refer to \cite{desv_parma} for complete proofs.
\medskip

We now establish three lemmas stating the (uniform w.r.t. $\var \in ]0,1[$) {\sl{a priori}} estimates for 
this solution $(u_A^\var,u_B^\var, v^\var)$.

\begin{lem}\label{prelim_u} Under the assumptions of Proposition \ref{theo_sing1}, the following (uniform w.r.t $\var \in ]0,1[$) estimates hold:
\begin{equation}
\sup_{0\le t \le T}\int_\Omega (u_{A}^\epsilon+u_{B}^\epsilon)(t)\le C_T; \hspace{1cm}  \|u_A^\epsilon+u_B^\epsilon\|_{L^{1+a}}\le C_T.
\end{equation}
\end{lem}

\begin{proof}[Proof of Lemma \ref{prelim_u}]
The quantity $u_A^\epsilon+u_B^\epsilon$ satisfies the equation
\begin{equation} \label{eq:3rd_u} 
\partial_{t} (u_{A}^\epsilon+u_{B}^\epsilon) - \Delta_x [M^\epsilon (u_{A}^\epsilon+u_{B}^\epsilon)] = [r_u-r_{a}(u_{A}^\epsilon+u_{B}^\epsilon)^a-r_{b}(v^\epsilon)^b](u_{A}^\epsilon+u_{B}^\epsilon)\le C_T,
\end{equation}
where $M^\epsilon = \frac{d_A u_A^\epsilon + (d_A + d_B) u_B^\epsilon}{u_{A}^\epsilon+u_{B}^\epsilon}$. 
We integrate w.r.t. space and time to get
\begin{equation}\label{L1_energy_estimate}
\sup_{0\le t \le T}\int_\Omega (u_{A}^\epsilon+u_{B}^\epsilon)(t) \le \int_\Omega (u_{A,in}^\epsilon+u_{B,in}^\epsilon) + C_T \le C_T,
\end{equation}
so that
\begin{equation}
\sup_{0\le t \le T}\int_\Omega (u_{A}^\epsilon+u_{B}^\epsilon)(t) + r_a\int_0^T\int_\Omega (u_{A}^\epsilon+u_{B}^\epsilon)^{1+a}
\le \int_\Omega (u_{A,in}^\epsilon+u_{B,in}^\epsilon) + r_u \int_0^T\int_\Omega (u_{A}^\epsilon+u_{B}^\epsilon) \le C_T.
\end{equation}
\end{proof}

\begin{lem}\label{prelim_v}  Under the assumptions of Proposition \ref{theo_sing1}, for all $1<q\le 1+ p_0/d$, the following (uniform w.r.t $\var \in ]0,1[$) estimates hold:
\begin{equation}
\|v^\epsilon\|_{L^\infty}\le C_T; \hspace{1cm} \|\nabla_x v^\epsilon\|_{L^{2q}}^2\le C_T \,(1+ \|(u_A^\epsilon+u_B^\epsilon)^d\|_{L^{q}}); \hspace{1cm}\|\partial_t v^\epsilon\|_{L^q}\le C_T\, (1+ \|(u_A^\epsilon+u_B^\epsilon)^d\|_{L^{q}}).
\end{equation}
\end{lem}

\begin{proof}[Proof of Lemma \ref{prelim_v}]
The first estimate is a consequence of the maximum principle for the equation satisfied (in 
the strong sense) by $v^\var$. More precisely, this maximum principle writes
\begin{equation}\label{vexa}
0 \le v^{\var}(t,x) \le \max \left(||v_{in}||_{L^{\infty}(\Omega)} + \var, \left[\frac{r_v}{r_c\,(c+1)}\right]^{1/c}\right)
\qquad {\hbox{ for a.e. }} \quad (t,x) \in \R_+\times\Omega.
\end{equation}

 We can then apply the maximal regularity result for the heat equation (satisfied by
 $v^\var$ when the reaction term is considered as given)
 in order to get the third estimate (note that we use here the assumption on $v_{in}$, since
$v_{in}^\var = v_{in} + \var$). The same bound also holds for $\pa_{x_i x_j}v^\epsilon$, so that interpolating with the first estimate, 
the second estimate holds.
\end{proof}

We now write down a (uniform w.r.t. $\var \in ]0,1[$) bound obtained thanks to the use 
of a Lyapounov-like (entropy) functional:
\medskip

\begin{lem}\label{entropy_estimates}
 Under the assumptions of Proposition \ref{theo_sing1},
for all $p \in ]1, p_0]$, the following inequalities hold:
\begin{equation}\label{es:supLp}\begin{split}
\sup_{t\in[0,T]}\int_\Omega (u_A^\epsilon+u_B^\epsilon)^{p}(t)
\le C_T\,(1+\|u_A^\epsilon+u_B^\epsilon\|_{L^{p+d}}^{p+d}),
\end{split}\end{equation}
\begin{equation}\label{eq:entropy_estimates}\begin{split}
\|u_A^\epsilon+u_B^\epsilon\|_{L^{p+a}}^{p+a}
\le C_T\,(1+\|u_A^\epsilon+u_B^\epsilon\|_{L^{p+d}}^{p+d}),
\end{split}\end{equation}
\begin{equation}\label{eq:entropy_inverse_estimates}\begin{split}
\|\nabla_x (u_A^\epsilon)^{p/2}\|_{L^2}^2
+ \|\nabla_x (u_B^\epsilon)^{p/2}\|_{L^2}^2
+\frac{1}{\epsilon}\| (h(v^\epsilon)u_A^\epsilon)^{p/2} - (k(v^\epsilon)u_B^\epsilon)^{p/2} \|_{L^2}^2
\le C_{T}\,(1+\|u_A^\epsilon+u_B^\epsilon\|_{L^{p+d}}^{p+d}).
\end{split}\end{equation}
\end{lem}

\begin{proof}[Proof of Lemma \ref{entropy_estimates}]
 We define the following entropy for any $p>0$ (with $p\neq 1$):
\begin{equation}\label{tw}
\mathscr{E}^\epsilon(t) = \int_\Omega h(v^\epsilon)^{p-1}\frac{(u_A^\epsilon)^p}{p}(t) + \int_\Omega k(v^\epsilon)^{p-1}\frac{(u_B^\epsilon)^p}{p}(t) \qquad (=: \mathscr{E}^\epsilon_A(t) + \mathscr{E}^\epsilon_B(t)).
\end{equation}
We compute the derivative (note that in the computation below all integrals lie in $L^1([0,T])$ thanks to the properties \eqref{es:eps_depend}; therefore the computation holds for a.e. $t\in[0,T]$):
\begin{equation}\label{two}
\begin{split}
\frac{\ud}{\ud t}\mathscr{E}_A^\epsilon(t) = \int_\Omega \partial_t\{ h(v^\epsilon)^{p-1}\frac{(u_A^\epsilon)^p}{p}\}(t)\\
= \frac{p-1}{p}\int_\Omega \partial_t v^\epsilon h'(v^\epsilon)h(v^\epsilon)^{p-2}(u_A^\epsilon)^p
+ \int_\Omega \partial_t u_A^\epsilon (u_A^\epsilon)^{p-1} h(v^\epsilon)^{p-1}\\
= \frac{p-1}{p}\int_\Omega \partial_t v^\epsilon h'(v^\epsilon)h(v^\epsilon)^{p-2}(u_A^\epsilon)^p
+ \int_\Omega [r_u-r_{a}(u_{A}^\epsilon+u_{B}^\epsilon)^a-r_{b}(v^\epsilon)^b](u_A^\epsilon)^{p} h(v^\epsilon)^{p-1}\\
+ \frac{1}{\epsilon} \int_\Omega [k(v^\epsilon)u_{B}^\epsilon-h(v^\epsilon)u_{A}^\epsilon] (u_A^\epsilon)^{p-1} h(v^\epsilon)^{p-1}
+ d_A \int_\Omega \Delta_x u_A^\epsilon (u_A^\epsilon)^{p-1} h(v^\epsilon)^{p-1},
\end{split} \end{equation}
where the last term is estimated 
 by integrating by part (and using the inequality $2 |ab|\le a^2 + b^2$) in the case when $p>1$:
\begin{equation}\begin{split}
 &  d_A \int_\Omega \Delta_x u_A^\epsilon (u_A^\epsilon)^{p-1} h(v^\epsilon)^{p-1}\\
=& -d_A \,(p-1)\int_\Omega |\nabla_x u_A^\epsilon|^2 (u_A^\epsilon)^{p-2} h(v^\epsilon)^{p-1}
- d_A \,(p-1) \int_\Omega \nabla_x u_A^\epsilon \cdot \nabla_x h(v^\epsilon) (u_A^\epsilon)^{p-1} h(v^\epsilon)^{p-2}\\
\le & - \frac{d_A}{2}\, (p-1) \int_\Omega |\nabla_x u_A^\epsilon|^2 (u_A^\epsilon)^{p-2} h(v^\epsilon)^{p-1}
+ \frac{d_A}{2} \, (p-1)\int_\Omega |\nabla_x h(v^\epsilon)|^2 (u_A^\epsilon)^{p} h(v^\epsilon)^{p-3}\\
=&- 2 \,d_A \,\frac{(p-1)}{p^2}\int_\Omega |\nabla_x (u_A^\epsilon)^{p/2}|^2 h(v^\epsilon)^{p-1}
+ \frac{(p-1)}{2}\, d_A\int_\Omega |\nabla_x v^\epsilon|^2 (u_A^\epsilon)^{p} (h'(v^\epsilon))^2 h(v^\epsilon)^{p-3}.
\end{split} \end{equation}

Similarly, we get for $u_B^\epsilon$ (still for a.e. $t\in[0,T]$), 
\begin{equation}
\begin{split}
\frac{\ud}{\ud t}\mathscr{E}_B^\epsilon(t)
\le \frac{p-1}{p}\int_\Omega \partial_t v^\epsilon k'(v^\epsilon)k(v^\epsilon)^{p-2}(u_B^\epsilon)^p\\
+ \int_\Omega [r_u-r_{a}(u_{B}^\epsilon+u_{B}^\epsilon)^a-r_{b}(v^\epsilon)^b](u_B^\epsilon)^{p} k(v^\epsilon)^{p-1}
- \frac{1}{\epsilon} \int_\Omega [k(v^\epsilon)u_{B}^\epsilon-h(v^\epsilon)u_{A}^\epsilon] (u_B^\epsilon)^{p-1} k(v^\epsilon)^{p-1}\\
- 2 (d_A+d_B) \frac{p-1}{p^2}\int_\Omega |\nabla_x (u_B^\epsilon)^{p/2}|^2 k(v^\epsilon)^{p-1}
+ \frac{p-1}{2} (d_A+d_B)\int_\Omega |\nabla_x v^\epsilon|^2 (u_B^\epsilon)^{p} (k'(v^\epsilon))^2 k(v^\epsilon)^{p-3}.
\end{split} 
\end{equation}

We add the two estimates and integrate w.r.t time to get (still for any $p>1$)
\begin{equation}\label{entropy_estimate}\begin{split}
\int_\Omega \bigg[h(v^\epsilon)^{p-1}\frac{(u_A^\epsilon)^p}{p}(T)
+ k(v^\epsilon)^{p-1}\frac{(u_B^\epsilon)^p}{p}(T) \bigg] \\
+ 2 \,d_A\, \frac{p-1}{p^2}\int_0^T\int_\Omega |\nabla_x (u_A^\epsilon)^{p/2}|^2 h(v^\epsilon)^{p-1}
+ 2 \,(d_A+d_B)\, \frac{p-1}{p^2}\int_0^T\int_\Omega |\nabla_x (u_B^\epsilon)^{p/2}|^2 k(v^\epsilon)^{p-1}\\
+\frac{1}{\epsilon} \int_0^T \int_\Omega [k(v^\epsilon)u_{B}^\epsilon-h(v^\epsilon)u_{A}^\epsilon] [(u_B^\epsilon)^{p-1} k(v^\epsilon)^{p-1}-(u_A^\epsilon)^{p-1} h(v^\epsilon)^{p-1}]\\
+\, r_{a} \int_0^T\int_\Omega (u_{A}^\epsilon+u_{B}^\epsilon)^a [(u_A^\epsilon)^{p} h(v^\epsilon)^{p-1}+(u_B^\epsilon)^{p} k(v^\epsilon)^{p-1}]\\
\le \int_\Omega \bigg[ h(v_{in}^\epsilon)^{p-1}\frac{(u_{A,in}^\epsilon)^p}{p}
+ k(v_{in}^\epsilon)^{p-1}\frac{(u_{A,in}^\epsilon)^p}{p} \bigg]\\
+\frac{p-1}{p}\int_0^T\int_\Omega \partial_t v^\epsilon [h'(v^\epsilon)h(v^\epsilon)^{p-2}(u_A^\epsilon)^p+k'(v^\epsilon)k(v^\epsilon)^{p-2}(u_B^\epsilon)^p]\\
+ r_u \,\int_0^T\int_\Omega \bigg[(u_A^\epsilon)^{p} h(v^\epsilon)^{p-1}+(u_B^\epsilon)^{p} k(v^\epsilon)^{p-1} \bigg]\\
 + \frac{p-1}{2}\int_0^T \int_\Omega [d_A(u_A^\epsilon)^{p} (h'(v^\epsilon))^2 h(v^\epsilon)^{p-3}+(d_A+d_B)(u_B^\epsilon)^{p} (k'(v^\epsilon))^2 k(v^\epsilon)^{p-3}]|\nabla_x v^\epsilon|^2 .
\end{split} 
\end{equation}

Let us estimate the right-hand side of inequality \eqref{entropy_estimate} under the assumptions of the lemma: the first term is finite since $p\le p_0$.
Thanks to the maximum principle for the density $v^\epsilon$ (obtained in Lemma \ref{prelim_v})
 and the regularity of the functions $h$ and $k$ in Assumption B, the terms $h(v^\epsilon)$, $h'(v^\epsilon)$ and $k(v^\epsilon)$, $k'(v^\epsilon)$ are uniformly bounded in $L^\infty$. We then can estimate the third term with H\"older's inequality.
Indeed,
$$  \bigg| \,\int_0^T\int_\Omega \bigg[(u_A^\epsilon)^{p} h(v^\epsilon)^{p-1}+(u_B^\epsilon)^{p} k(v^\epsilon)^{p-1} \bigg] \,\bigg| $$
$$ \le || h(v^\epsilon)^{p-1} + k(v^\epsilon)^{p-1} ||_{L^{\infty}}\, || u_A^\epsilon +
u_B^\epsilon ||_{L^p}^p \le C_T\, (1+  \|u_A^\epsilon+u_B^\epsilon\|_{L^{p+d}}^{p+d}) . $$
 The second and the last terms are estimated thanks to H\"older's inequality and bounds given by Lemma \ref{prelim_v}. 
More precisely, for the second term, we get
\begin{equation}\label{dt_v_in_entropy_estimate}
\begin{split}
\left| \frac{p-1}{p}\int_0^T\int_\Omega \partial_t v^\epsilon [h'(v^\epsilon)h(v^\epsilon)^{p-2}(u_A^\epsilon)^p+k'(v^\epsilon)k(v^\epsilon)^{p-2}(u_B^\epsilon)^p]\right|\\
\le  \| h'(v^\epsilon)h(v^\epsilon)^{p-2}+k'(v^\epsilon)k(v^\epsilon)^{p-2}\|_{L^{\infty}}
\| \partial_t v^\epsilon \|_{L^{1+p/d}}
\| (u_A^\epsilon+u_B^\epsilon)^p \|_{L^{1+d/p}}
\le C_T\,(1+\|u_A^\epsilon+u_B^\epsilon\|_{L^{p+d}}^{p+d}),
\end{split} 
\end{equation}
and for the last term, we get
\begin{equation}\label{nabla_v_in_entropy_estimate}
\begin{split}
\left| \frac{p-1}{2}\int_0^T \int_\Omega [d_A(u_A^\epsilon)^{p} (h'(v^\epsilon))^2 h(v^\epsilon)^{p-3}+(d_A+d_B)(u_B^\epsilon)^{p} (k'(v^\epsilon))^2 k(v^\epsilon)^{p-3}]|\nabla_x v^\epsilon|^2 \right|\\
\le (p/2)\, \| h'(v^\epsilon)^2h(v^\epsilon)^{p-3}+k'(v^\epsilon)^2k(v^\epsilon)^{p-3}\|_{L^{\infty}}
\| \, |\nabla_x v^\epsilon|^2 \|_{L^{1+p/d}}
\| (u_A^\epsilon+u_B^\epsilon)^p \|_{L^{1+d/p}}
\le C_T\,(1+\|u_A^\epsilon+u_B^\epsilon\|_{L^{p+d}}^{p+d}),
\end{split} 
\end{equation}
thanks to Lemma \ref{prelim_v}.
\smallskip

The terms of the left-hand side of \eqref{entropy_estimate} being all nonnegative, they are all 
bounded  by the quantity $(1+\|u_A^\epsilon+u_B^\epsilon\|_{L^{p+d}}^{p+d})$. We then obtain the estimates announced in the lemma
by using the lower bound of $h$ and $k$ (remember Assumption B), and the 
following elementary inequality for all positive $x$, $y$ : 
$(x-y)\,(x^{p-1} - y^{p-1}) \ge C_{p}\, |x^{p/2}-y^{p/2}|^2$, where $C_{p}>0$ is a constant 
depending on $p$ (remember that $p \in ]1,p_0]$).
\end{proof}
\medskip

We now turn back to the proof of Proposition \ref{theo_sing1}.
\medskip

As a first consequence of Lemma \ref{entropy_estimates}, we can improve the Lebesgue space in which we get a uniform (w.r.t. $\var$) estimate for $u_A^\epsilon+u_B^\epsilon$. Taking $p=p_0$ in \eqref{eq:entropy_estimates} and using H\"older's inequality (remember that $d<a$),
we see that
$$ \|u_A^\epsilon+u_B^\epsilon\|_{L^{p_0+a}}^{p_0+a}\le C_T\, (1+  \|u_A^\epsilon+u_B^\epsilon\|_{L^{p_0+d}}^{p_0+d} ) \le  C_T\, (1+ 
\|u_A^\epsilon+u_B^\epsilon\|_{L^{p_0+a}}^{p_0+d} ), $$
so that
\begin{equation}\label{est:L_(p+a)_u}
\|u_A^\epsilon+u_B^\epsilon\|_{L^{p_0+a}}\le C_T.
\end{equation}

Let us combine estimate (\ref{est:L_(p+a)_u}) and Lemma \ref{prelim_v} with $q=1 + p_0/d>1$ to get
\begin{equation}\label{est:gradv}
\|\nabla_x v^\epsilon\|_{L^{2\,(1 + p_0/d)}}^2\le C_T,\qquad\|\partial_t v^\epsilon\|_{L^{1 + p_0/d}}
\le C_T.
\end{equation}
Then, 
from Aubin's lemma (see Theorem 5 in \cite{sim}), we can extract a subsequence - still
 called $(v^\epsilon)_\epsilon$ - which converges towards a limit $v$  
a.e. :
\begin{equation}\label{convergence_v}
v^\epsilon(t,x) \rightarrow v(t,x) \hspace{1cm} \text{almost everywhere on }[0,T]\times\Omega,
\end{equation}
and such that 
\begin{equation}\label{convergence_grad_v}
 \nabla_x v^\epsilon  \rightharpoonup  \nabla_x v  \qquad {\hbox{ weakly in }} \qquad L^{2\,(1 + p_0/d)} 
 \quad
{\hbox{ (and therefore in) }} \quad L^1 . 
\end{equation}
Thanks to this passage to the limit, the function $v$ automatically lies in $L^{\infty}$, and is nonnegative. Passing to the limit in estimate (\ref{vexa}), we get estimate (\ref{vex}).
 Finally, $\nabla_x v \in L^{2\,(1 + p_0/d)}$.
\smallskip

Recall now eq. \eqref{eq:3rd_u} for $u_{A}^\epsilon+u_{B}^\epsilon$. Notice that the reaction term in \eqref{eq:3rd_u} is uniformly bounded in $L^{\lambda}$ with $\lambda= \frac{p_0 +a}{1+a}>1$, thanks to estimate (\ref{est:L_(p+a)_u}).
 As a consequence, $\partial_t (u_{A}^\epsilon+u_{B}^\epsilon)$ in \eqref{eq:3rd_u} is uniformly bounded in $L^{\lambda}([0,T],W^{-2,\lambda})$. Furthermore, let us choose some $p$ in the interval $]1, p_0[$ and $\zeta=\zeta(p) \in ]0,1[$ such that $(2-p)\,\frac{1+\zeta}{1-\zeta} < p_0+a$ . Then for $C=A$ or $B$, H\"older's inequality
 implies
\begin{equation}\label{twn}\begin{split}
\|\, |\nabla_x u_{C}^\epsilon|^{1 + \zeta} \, \|_{L^1} \le
 \int_0^T\int \left[ |u_{C}^\epsilon|^{p/2-1}\, |\nabla_x u_{C}^\epsilon| \right]^{1+\zeta} \,\,
   |u_{C}^\epsilon|^{(1 - p/2)\, (1+\zeta)} \\
\le \left( \int_0^T\int \left[ |u_{C}^\epsilon|^{p/2-1}\, |\nabla_x u_{C}^\epsilon| \right]^{2} \right)^{\frac{1+\zeta}{2}} \,\, 
\left( \int_0^T\int \left[ |u_{C}^\epsilon|^{(2-p)\,\frac{1+\zeta}{1-\zeta}}
\right] \right)^{\frac{1-\zeta}{2}} \\
\le (2/p)^{1+\zeta} \, ||  \nabla_x (u_{C}^\epsilon)^{p/2}  ||_{L^2}^{1+\zeta}
\,\, 
\left( \int_0^T\int \left[ |u_{C}^\epsilon|^{(2-p)\,\frac{1+\zeta}{1-\zeta}}
\right] \right)^{\frac{1-\zeta}{2}} \le C_T,
\end{split}\end{equation}
thanks to Lemma \ref{entropy_estimates} and estimate (\ref{est:L_(p+a)_u}).
\par
 We therefore can apply Aubin's lemma to extract a subsequence (still called $(u_{A}^\epsilon+u_{B}^\epsilon)_\epsilon$) which converges towards a limit $u$  
a.e.: 
\begin{equation}\label{convergence_u}
u_A^\epsilon(t,x)+u_B^\epsilon(t,x) \rightarrow u(t,x) \hspace{1cm} \text{almost everywhere on }[0,T]\times\Omega .
\end{equation}
Thanks to estimate (\ref{est:L_(p+a)_u}) and Fatou's lemma, we know that $u \in L^{a+p_0}$.
 Moreover, $u \ge 0$ a.e. thanks to the passage to the limit a.e., and $ \nabla_x u \in L^{1 +\zeta}$
for some $\zeta>0$ small enough, thanks to estimate (\ref{twn}).
\medskip

We now use the following elementary inequality: for any $p\in ]0,2[$, there exists
a constant $C_p>0$ (which depends only on $p$)
such that
 \begin{equation} \label{forp}
\forall x \in \R_+, \qquad |x-1|\le C_p\, |x^{p/2}-1| \times |x^{1-p/2}+1|.
\end{equation}
 
Taking $p$ in the interval $]1, \min\{p_0, 2\}[$, we see that
\begin{equation} \label{starstar}\begin{split}
\int_0^T\int_\Omega |k(v^\epsilon)u_{B}^\epsilon-h(v^\epsilon)u_{A}^\epsilon|
\le 
C_p \int_0^T\int_\Omega |(u_B^\epsilon k(v^\epsilon))^{p/2}-(u_A^\epsilon h(v^\epsilon))^{p/2}| \times [(u_B^\epsilon k(v^\epsilon))^{1-p/2}+(u_A^\epsilon h(v^\epsilon))^{1-p/2}] \\
\le
C_p \left(\int_0^T\int_\Omega |(u_B^\epsilon k(v^\epsilon))^{p/2}-(u_A^\epsilon h(v^\epsilon))^{p/2}|^2\right)^{1/2}
\times \left(\int_0^T\int_\Omega [(u_B^\epsilon k(v^\epsilon))^{1-p/2}+(u_A^\epsilon h(v^\epsilon))^{1-p/2}]^2 \right)^{1/2}
\le \sqrt{\epsilon}\,  C_T,
\end{split}\end{equation}
thanks to Lemma \ref{entropy_estimates}. Then,
$u_B^\epsilon \,k(v^\epsilon)- u_A^\epsilon \,h(v^\epsilon)$ converges to $0$
 in $L^1$, and therefore, up to a subsequence,
\begin{equation}\label{convergence_hu_A-ku_B}
h(v^\epsilon(t,x))\,u_A^\epsilon(t,x)-k(v^\epsilon(t,x))\,u_B^\epsilon(t,x) \rightarrow 0 \hspace{1cm} \text{almost everywhere on }[0,T]\times\Omega.
\end{equation}

Thanks to the convergences \eqref{convergence_v}, \eqref{convergence_u} and \eqref{convergence_hu_A-ku_B},
 we can compute

\begin{equation}\label{convergence_u_A}
u_A^\epsilon(t,x)=\frac{k(v^\epsilon)\,(u_A^\epsilon+u_B^\epsilon) + [h(v^\epsilon)\,u_A^\epsilon-k(v^\epsilon)\,u_B^\epsilon]}{h(v^\epsilon)+k(v^\epsilon)} \rightarrow \frac{k(v)\,u}{h(v)+k(v)}=:u_A(t,x) \hspace{1mm} \text{almost everywhere on }[0,T]\times\Omega,
\end{equation}
and similarly
\begin{equation}\label{convergence_u_B}
u_B^\epsilon(t,x)=\frac{h(v^\epsilon)\,(u_A^\epsilon+u_B^\epsilon) - [h(v^\epsilon)\,u_A^\epsilon-k(v^\epsilon)\,u_B^\epsilon]}{h(v^\epsilon)+k(v^\epsilon)} \rightarrow \frac{h(v)\,u}{h(v)+k(v)}=:u_B(t,x) \hspace{1mm} \text{almost everywhere on }[0,T]\times\Omega.
\end{equation}
Up to another extraction, we see that, thanks to estimate (\ref{twn}), for $C=A,B$,
 \begin{equation}\label{convergence_grad_u}
 \nabla_x u_C^\epsilon  \rightharpoonup  \nabla_x u_C 
\qquad {\hbox{ weakly in }} \qquad L^1. 
\end{equation}

Extracting again subsequences, we can perform this proof on $[0,2T]$, $[0,3T]$, ..., so that by Cantor's diagonal argument, 
\begin{equation}\label{convul}
u_A^\epsilon(t,x) \to u_A(t,x), \quad  u_B^\epsilon(t,x) \to u_B(t,x), \quad
v^\epsilon(t,x) \to v(t,x)
\end{equation}
 for a.e. $(t,x)\in \R_+ \times\Omega$,
 where $u_A,u_B,v$ are
 defined on $\R_+\times\Omega$. It is clear that $u_A,u_B,v \ge 0$ a.e.
 Remembering the definition of $u_A$ and $u_B$, we also see that
$h(v)\, u_A = k(v)\,u_B$ a.e., and $u_A,u_B \in L^{a+p_0}$. Finally, we recall that
$v \in L^{\infty}$.
\medskip

Let us now show \eqref{es:V2}. Thanks to the uniform (in $\var$) estimates \eqref{es:supLp}, \eqref{eq:entropy_inverse_estimates} and \eqref{est:L_(p+a)_u}, we have for all $p\in]1,p_0]$,
\begin{equation}
\sup_{t\in [0,T]} \int_\Omega u^{p} (t) < +\infty \qquad \text{and} \qquad \int_0^T \int_\Omega |\nabla_x u_A^{p/2}|^2+|\nabla_x u_B^{p/2}|^2<+\infty,
\end{equation}
where we have used Fatou's lemma for the first inequality and Kakutani's Theorem applied to the reflexive space $L^2$ for the second inequality.
Remembering that $u=\frac{h(v)+k(v)}{k(v)}\, u_A$, we can see that for all $p\in]1,p_0]$,
\begin{align*}
\nabla_x \left(u^{p/2}\right) = \stackrel{\in L^{2(1+a/p_0)}}{\overbrace{\phantom{\Bigg|} u_A^{p/2} }} \stackrel{\in L^\infty}{\overbrace{\left[\left(\frac{h+k}{k}\right)^{p/2}\right]'(v)}} \stackrel{\in L^{2(1+p_0/d)}}{\overbrace{\phantom{\Bigg|} \nabla_x v}} + \stackrel{\in L^\infty}{\overbrace{\left(\frac{h(v)+k(v)}{k(v)}\right)^{p/2}}} \,\stackrel{\in L^2}{\overbrace{\phantom{\Bigg|} \nabla_x \left(u_A^{p/2}\right)}} \;\in\; L^2.
\end{align*}

In order to conclude the proof of Proposition \ref{theo_sing1}, it only remains to check that $(u,v)=(u_A+u_B,v)$ is a weak solution of \eqref{sku1}--\eqref{sku4} in the sense of Definition~\ref{defimer}.
\par
 Let $\psi_1$, $\psi_2 \in C^1_c(\R_+\times {\ov{\Omega}})$ be test functions.
 Multiplying all terms of the two first equations of \eqref{eq:3rd_u} by $\psi_1$, multiplying all terms of equation \eqref{eq:3rd} by $\psi_2$,
 and integrating on $\R_+\times\Omega$, we get
\begin{align}\label{wee1}
 - \int_0^{\infty}\int_{\Omega} \pa_t \psi_1\, (u_A^\epsilon+u_B^\epsilon) - \int_{\Omega} \psi_1(0,\cdot)\, (u_{A,in}^\epsilon+u_{B,in}^\epsilon) 
+ \int_0^{\infty}\int_{\Omega} \nabla_x \psi_1 \cdot \nabla_x( M^\epsilon\, (u_A^\epsilon+u_B^\epsilon))\\ =  \int_0^{\infty}\int_{\Omega} \psi_1\, (u_A^\epsilon+u_B^\epsilon)\, (r_u - r_{a}\, (u_A^\epsilon+u_B^\epsilon)^a - r_{b}\,(v^\epsilon)^b), 
\end{align}
\begin{equation}\label{wee2}
 - \int_0^{\infty}\int_{\Omega} \pa_t \psi_2\, v^\epsilon - \int_{\Omega} \psi_2(0,\cdot)\, v_{in}^\epsilon 
+ d_v\int_0^{\infty}\int_{\Omega} \nabla_x \psi_2 \cdot \nabla_x v^\epsilon = \int_0^{\infty}\int_{\Omega} \psi_2\,
 v^\epsilon\, (r_v - r_{c}\,(v^\epsilon)^c - r_d\, (u_A^\epsilon+u_B^\epsilon)^d). 
 \end{equation}
Note that thanks to (\ref{convul}), 
 $$\pa_t \psi_1\, (u_A^\epsilon+u_B^\epsilon) \to (\pa_t \psi_1)\, u, $$
for a.e. $(t,x)\in \R_+ \times\Omega$, and $\psi_1\, (u_A^\epsilon+u_B^\epsilon)$ is bounded
(uniformly w.r.t. $\var \in ]0,1[$) in $L^{p_0+a}$ thanks to (\ref{est:L_(p+a)_u}), so that
\begin{equation}\label{weec1}
 - \int_0^{\infty}\int_{\Omega} \pa_t \psi_1\, (u_A^\epsilon+u_B^\epsilon)  \to
 - \int_0^{\infty}\int_{\Omega} (\pa_t \psi_1)\, u . 
 \end{equation}
In the same way, since $v^\var$ is uniformly bounded w.r.t. $\var\in ]0,1[$, (\ref{convul}) and (\ref{convergence_grad_v})
imply that
\begin{equation}\label{weec2}
  - \int_0^{\infty}\int_{\Omega} \pa_t \psi_2\, v^\epsilon \to - \int_0^{\infty}\int_{\Omega} \pa_t \psi_2\, v, \quad  d_v\int_0^{\infty}\int_{\Omega} \nabla_x \psi_2 \cdot \nabla_x v^\epsilon
\to
d_v\int_0^{\infty}\int_{\Omega} \nabla_x \psi_2 \cdot \nabla_x v,
\end{equation}
 Then, we observe that $\psi_1\, (u_A^\epsilon+u_B^\epsilon)\, (r_u - r_{a}\, (u_A^\epsilon+u_B^\epsilon)^a - r_{b}\,(v^\epsilon)^b)$ is bounded (uniformly w.r.t. $\var \in ]0,1[$) in $L^{\frac{p_0+a}{1+a}}$
thanks to (\ref{est:L_(p+a)_u}), so that (\ref{convul}) implies that
\begin{equation}\label{weec3}
 \int_0^{\infty}\int_{\Omega} \psi_1\, (u_A^\epsilon+u_B^\epsilon)\, (r_u - r_{a}\, (u_A^\epsilon+u_B^\epsilon)^a - r_{b}\,(v^\epsilon)^b) \to
\int_0^{\infty}\int_{\Omega} \psi_1\, u\, (r_u - r_{a}\, u^a - r_{b}\,v^b) . 
\end{equation}
In the same way, $\psi_2\,
 v^\epsilon\, (r_v - r_{c}\,(v^\epsilon)^c - r_d\, (u_A^\epsilon+u_B^\epsilon)^d)$ 
 is bounded (uniformly w.r.t. $\var \in ]0,1[$) in $L^{\frac{p_0+a}{d}}$, so that
 \begin{equation}\label{weec4}
 \int_0^{\infty}\int_{\Omega} \psi_2\,
 v^\epsilon\, (r_v - r_{c}\,(v^\epsilon)^c - r_d\, (u_A^\epsilon+u_B^\epsilon)^d)
 \to \int_0^{\infty}\int_{\Omega} \psi_2\,
 v\, (r_v - r_{c}\,v^c - r_d\, u^d).
 \end{equation}
 According to the definition of $u_{A,in}^\epsilon$ and $u_{B,in}^\epsilon$, it is clear that
 $u_{A,in}^\epsilon \to u_{A,in}$ a.e. on $\Omega$, and $u_{B,in}^\epsilon \to u_{B,in}$ a.e.
 on $\Omega$, so that $u_{A,in}^\epsilon + u_{B,in}^\epsilon \to u_{in}$ a.e. on $\Omega$. But
 $u_{in} \in L^{p_0}(\Omega)$, so that $u_{A,in}$ and $u_{B,in}$ also lie in $L^{p_0}(\Omega)$,
and $u_{A,in}^\epsilon,  u_{B,in}^\epsilon$ are bounded (uniformly w.r.t. $\var \in ]0,1[$) in 
$L^{p_0}(\Omega)$.
Then
\begin{equation}\label{weec5}
\int_{\Omega} \psi_1(0,\cdot)\, (u_{A,in}^\epsilon+u_{B,in}^\epsilon)  \to
\int_{\Omega} \psi_1(0,\cdot)\, u .
\end{equation}
In the same way, observing that $v_{in}^\var$ is bounded (uniformly w.r.t. $\var \in ]0,1[$) in $L^{\infty}(\Omega)$, we see that
\begin{equation}\label{weec6}
 - \int_{\Omega} \psi_2(0,\cdot)\, v_{in}^\epsilon \to - \int_{\Omega} \psi_2(0,\cdot)\, v_{in} . 
 \end{equation}
It remains to study the convergence of $\nabla_x \psi_1\cdot \nabla_x [M^\var\,(u_A^\var + u_B^\var)]$.
But $M^\var\,(u_A^\var + u_B^\var) = d_A\, u_A^\var + (d_A+d_B)\, u_B^\var$, so that
$M^\var\,(u_A^\var + u_B^\var) \to d_A\, u_A + (d_A+d_B)\, u_B = d_A \,u + d_B\,u_B 
= d_A \, u + d_B \,\frac{h(v)\,u}{h(v)+k(v)} = (d_u+ \phi(v))\,u$ a.e. on $\R_+ \times \Omega$.
Then, using the convergence (\ref{convergence_grad_u}),
\begin{equation}\label{weec7} \int_0^{\infty}\int_{\Omega} \nabla_x \psi_1\cdot \nabla_x [M^\var\,(u_A^\var + u_B^\var)] \to
 \int_0^{\infty}\int_{\Omega} \nabla_x \psi_1\cdot  \nabla_x [(d_u+ \phi(v))\,u] . 
 \end{equation}
Note that this automatically implies the estimate $\nabla_x (\phi(v)\,u) \in L^1$. It is however 
possible to directly get it by
 using estimate \eqref{twn} and the fact that $\nabla_x v\in L^{2(1+p_0/d)}$.
Indeed, one can get a slightly better estimate:
\begin{align*}
\nabla_x \left[ (d_u + \phi(v)) u\right] = \stackrel{\in L^{\infty}}{\overbrace{\phantom{\big|} \left(d_u + \phi(v) \right)}} \stackrel{\in L^{1+\zeta}}{\overbrace{\phantom{\big|} \nabla_x u}} + \stackrel{\in L^{p_0+a}}{\overbrace{\phantom{\big|} u}} \,\stackrel{\in L^{2(1+p_0/d)}}{\overbrace{\phantom{\big|} \nabla_x \phi(v)}} \;\in\; L^{1+\zeta'}  \qquad \text{for some }\zeta,\,\zeta'>0.
\end{align*}
 This concludes the proof of Proposition \ref{theo_sing1}.
\end{proof}
\medskip

We now turn to the

\begin{proof}[Proof of Proposition \ref{theo_sing2}]
 As in Proposition \ref{theo_sing1}, we recall that
for a given $\var\in ]0,1[$, standard theorems for reaction-diffusion equations imply the existence of a (global, nonnegative for each component) strong solution $(u_A^\var,u_B^\var, v^\var)$ in the sense of Definition \ref{defimerd}, to system
(\ref{eq:3rd}), (\ref{Neumann_3RD}), (\ref{initial_3RD}). Moreover, properties \eqref{es:eps_depend} hold  with $p_0=2$.
 We again refer to \cite{desv_parma} for complete proofs.
\medskip

Note first that the estimates of Lemmas \ref{prelim_u} and \ref{prelim_v}
still hold under the assumptions of Proposition \ref{theo_sing2},
with $p_0=2$ in the case of Lemma  \ref{prelim_v}. 
\medskip

More precisely,  the following (uniform w.r.t. $\var\in ]0,1[$)
estimates hold, the proofs being identical to those 
of Lemmas \ref{prelim_u} and  \ref{prelim_v}:
\begin{equation} \label{prelim_upr}
\sup_{0\le t \le T}\int_\Omega (u_{A}^\epsilon+u_{B}^\epsilon)(t)\le C_T; \hspace{1cm}  \|u_A^\epsilon+u_B^\epsilon\|_{L^{1+a}}\le C_T,
\end{equation}
and for all $1<q\le 1+ 2/d$, 
\begin{equation} \label{prelim_vpr}
\|v^\epsilon\|_{L^\infty}\le C_T; \hspace{1cm} \|\nabla_x v^\epsilon\|_{L^{2q}}^2\le C_T \,(1+ \|(u_A^\epsilon+u_B^\epsilon)^d\|_{L^{q}}); \hspace{1cm}\|\partial_t v^\epsilon\|_{L^q}\le C_T\, (1+ \|(u_A^\epsilon+u_B^\epsilon)^d\|_{L^{q}}).
\end{equation}
In fact, estimate (\ref{vexa}) still holds (it is the explicit version of the first part of (\ref{prelim_vpr})).
\medskip

As a consequence, for all $p \in ]0,1[$, taking $q= 1+p/d \le 1 + 2/d$, 
\begin{equation} \label{aappdd}
 \|\, |\nabla_x v^\epsilon|^2 \|_{L^{1+p/d}}
\le C_T \,(1+ \|(u_A^\epsilon+u_B^\epsilon)^d\|_{L^{1+p/d}});
 \hspace{1cm}\|\partial_t v^\epsilon\|_{L^{1+p/d}}\le C_T\, (1+ \|(u_A^\epsilon+u_B^\epsilon)^d\|_{L^{1+p/d}}).
\end{equation}

We now introduce a duality lemma in the spirit of the one used in \cite{CDF}:
\medskip

\begin{lem}\label{le:duality}
We consider $T>0$, $\Omega$ a bounded regular open set of $\R^N$ ($N\in \N^*$), and a function $M: = M(t,x)$ satisfying
\begin{equation}
0<m_0\le M (t,x) \le m_1 \qquad \text{for } t\in [0,T]\text{ and } x\in \Omega,
\end{equation}
for some constants $m_0, m_1>0$. Then, one can find $p^\ast>2$ such that for all $r\in[2,p^\ast[$, there exists a constant $C_T>0$ depending only on $\Omega$, $N$, $T$, and the constants $m_0$, $m_1$, $r$, such that for any initial datum $u_{in}$ in $L^2(\Omega)$ and any $K>0$, all nonnegative strong solutions $u\in L^r([0,T]\times\Omega)$ of the system
\begin{equation}\label{eq:duality_lemma}\left\{\begin{aligned}
& \partial_{t}u - \Delta_x (Mu) \leq K \text{ in }[0,T]\times\Omega,\\
& u(0,x)=u_{in}(x) \text{ in } \Omega,\\
& \nabla_x (Mu)(t,x)\cdot n(x) = 0 \text{ on } [0,T]\times\partial\Omega,
\end{aligned}\right.\end{equation}
satisfy
\begin{equation}\label{dual_estimate}
\|u\|_{L^r([0,T]\times\Omega)} \le C_T\, \left( \|u_{in}\|_{L^2(\Omega)}+K \right).
\end{equation}
\end{lem}

\begin{proof}[Proof of Lemma \ref{le:duality}]
It relies on the study of the dual problem
\begin{equation}\label{eq:dual_problem}\left\{\begin{aligned}
& \partial_{t}v + M \Delta_x v = - f \text{ in }[0,T]\times\Omega,\\
& v(T,x)=0 \text{ in } \Omega,\\
& \nabla_x v(t,x)\cdot n(x) = 0 \text{ on } [0,T]\times\partial\Omega,
\end{aligned}\right.\end{equation}
for $f$ a nonnegative function in $L^{r'}([0,T]\times\Omega)$, with $\frac1r + \frac1{r'} =1$.
\medskip

Using the notations of \cite{CDF}, we define the constant $C_{m,q}>0$ for $m>0$, $q\in]1,2]$ as the best constant in the parabolic estimate
\begin{equation}\label{heat_estimate}
\|\Delta_x w\|_{L^q([0,T]\times\Omega)} \le C_{m,q}\, \|g\|_{L^q([0,T]\times\Omega)},
\end{equation}
where $g$ is any function in ${L^q([0,T]\times\Omega)}$ and $w$ is the solution of the backward heat equation \begin{equation}\label{eq:forward_heat}\left\{\begin{aligned}
& \partial_{t}w + m\, \Delta_x w = g \text{ in }[0,T]\times\Omega,\\
& w(T,x)=0 \text{ in } \Omega,\\
& \nabla_x w(t,x)\cdot n(x) = 0 \text{ on } [0,T]\times\partial\Omega.
\end{aligned}\right.\end{equation}

 Let $r\ge 2$, $q=r'\le2$ and let $f$ be any smooth function defined on $[0,T]\times\Omega$. We consider the solution $v$ of system \eqref{eq:dual_problem}. Notice that thanks to the minimum principle, $v$ is nonnegative. Then, from Lemma 2.2 and Remark 2.3 in \cite{CDF}, there exists a constant $C_T$ depending only on $\Omega$, $N$, $T$ and $m_0$, $m_1$, $q$ such that $v$ satisfies
\begin{equation}\label{estimate_cdf_delta}
\|\Delta_x v\|_{L^q} \le C_T\, \|f\|_{L^q},
\end{equation}
and
\begin{equation}\label{estimate_cdf_initial}
\|v(0,\cdot)\|_{L^2(\Omega)} \le C_T \, \|f\|_{L^q},
\end{equation}
provided that $q>2\frac{N+2}{N+4}$ and
\begin{equation}\label{condition_constant}
C_{\frac{m_0+m_1}{2},q} \, \frac{m_1-m_0}{2} < 1.
\end{equation}
Let us first assume that condition \eqref{condition_constant} holds
 for some fixed $q\in]2\frac{N+2}{N+4},2]$. Then we compute (for a.e. $t\in[0,T]$)
\begin{equation}
\frac{d}{dt}\int_{\Omega}u(t)v (t) \le K \int_\Omega v(t) - \int_\Omega u(t)f (t),
\end{equation}
so that integrating w.r.t. time, and using the condition $v(T,\cdot)=0$,
\begin{equation}
\int_0^T \int_\Omega uf \le K \int_0^T\int_\Omega v + \int_{\Omega}u_{in}\, v(0,\cdot).
\end{equation}
The first term is estimated with \eqref{estimate_cdf_delta}:
\begin{equation}
\int_0^T\int_\Omega v = - \int_0^T\int_\Omega \int_t^T \partial_t v = \int_0^T\int_\Omega \int_t^T \left(f + M \Delta_x v \right) 
\end{equation}
$$\le T \left( \int_0^T\int_\Omega [f + m_1 |\Delta_x v|] \right)
\le T^{1+1/p} \, |\Omega|^{1/p}\, \left( 1+m_1 C_T \right) \|f\|_{L^q}, $$
and the second term with \eqref{estimate_cdf_initial}:
\begin{equation}
\int_{\Omega}u_{in}\,v(0,\cdot) \le \|u_{in}\|_{L^2(\Omega)}\, \|v(0,\cdot)\|_{L^2(\Omega)} \le C_T\, \|f\|_{L^q} \,
\|u_{in}\|_{L^2(\Omega)}.
\end{equation}
Recombining those estimates, we get
\begin{equation}
\int_0^T \int_\Omega u\, f \le C_T \left( K + \|u_{in}\|_{L^2(\Omega)}\right) \|f\|_{L^q},
\end{equation}
which, by duality, gives estimate \eqref{dual_estimate} (note that it is sufficient to show the previous bound for smooth $f$, since all functions of $L^q$ can be approximated by such smooth functions in the $L^q$ norm).

It remains to check that there exists an interval $[2,p^\ast[$ in which any $r$ satisfies condition \eqref{condition_constant} with $q=r'$. This is done in \cite{CDF}.
\end{proof}
\medskip

We now come back to the proof of Proposition \ref{theo_sing2}.
\medskip

As in the proof of Proposition \ref{theo_sing1}, we add the two equations and get
\begin{equation}\label{newse}
 \partial_{t} (u_{A}^\epsilon+u_{B}^\epsilon) - \Delta_x [M^\epsilon (u_{A}^\epsilon+u_{B}^\epsilon)] = [r_u-r_{a}(u_{A}^\epsilon+u_{B}^\epsilon)^a-r_{b}(v^\epsilon)^b](u_{A}^\epsilon+u_{B}^\epsilon),
\end{equation} 
with
$$M^\epsilon = \frac{d_A u_A^\epsilon + (d_A + d_B) u_B^\epsilon}{u_{A}^\epsilon+u_{B}^\epsilon}. $$
Then $ d_A \le M^\epsilon  \le d_A + d_B$, and 
$$ [r_u-r_{a}(u_{A}^\epsilon+u_{B}^\epsilon)^a-r_{b}(v^\epsilon)^b]
\,(u_{A}^\epsilon+u_{B}^\epsilon) \le \sup_{w \ge 0} [r_u-r_{a}\,w^a]\,w = \bigg( \frac{r_u}{(1+a)\,r_a} \bigg)^{1/a} := K. $$
We can apply Lemma \ref{le:duality} to eq. (\ref{newse}), with $M$ replaced by $M^\var$, $u$ replaced by $u_A^\var + u_B^\var$, and $u_{in}$ replaced by $u_{A,in}^\var + u_{B,in}^\var$.
Note that for any $\var>0$, $u_A^\var + u_B^\var$ is a strong solution of eq. (\ref{newse}).
\smallskip

 Lemma \ref{le:duality} implies that for some $p^{\ast}>2$,
\begin{equation}\label{star}
|| u_A^\var + u_B^\var||_{L^{p^\ast}} \le C_T .
\end{equation}

Using estimates (\ref{star}) and (\ref{prelim_vpr}), we see that
when $p \in ]0, p^\ast - d]$,
\begin{equation}\label{starbi}
 \|\,\nabla_x v^\epsilon \|_{L^{2\,(1+p/d)}}
\le C_T;
 \hspace{1cm}\|\partial_t v^\epsilon\|_{L^{1+p/d}}\le C_T .
\end{equation}

Thanks to estimates (\ref{prelim_vpr}), (\ref{starbi}), we can extract from $(v^\epsilon)_{\epsilon>0}$ a subsequence 
(still denoted $(v^\epsilon)_{\epsilon>0}$) which converges a.e. towards some $v \in L^{\infty}$, and such that
$\nabla_x v^\epsilon$ converges weakly in $L^{2\,(1+p/d)}$ (and therefore in $L^1$) towards $\nabla_x v$.
\medskip

Recalling definition (\ref{tw}) and computation (\ref{two}) in the case when $p \in ]0, \inf(1, p^\ast - d) [$,
 we use the inequality (for a.e. $t\in[0,T]$)
$$  - d_A \int_\Omega \Delta_x u_A^\epsilon (u_A^\epsilon)^{p-1} h(v^\epsilon)^{p-1}
= - 4\,d_A \,\frac{(1-p)}{p^2}\int_\Omega |\nabla_x [(u_A^\epsilon)^{p/2}] |^2 h(v^\epsilon)^{p-1}$$
$$ - d_A\,(1-p) \int_\Omega  (u_A^\epsilon)^{p-1} h'(v^\epsilon)\, h(v^\epsilon)^{p-2}\,
\nabla_x u_A^\epsilon \cdot \nabla_x v^\epsilon $$
$$\le - 2 \,d_A \,\frac{(1-p)}{p^2}\int_\Omega |\nabla_x [(u_A^\epsilon)^{p/2}] |^2 h(v^\epsilon)^{p-1} + \frac{(1-p)}{2}\, d_A\int_\Omega |\nabla_x v^\epsilon|^2 (u_A^\epsilon)^{p} 
(h'(v^\epsilon))^2 h(v^\epsilon)^{p-3}, $$
and the corresponding inequality for $u_B^\var$ (with $d_A$ replaced by $d_A+d_B$)
and get the estimate
\begin{equation}\label{entropy_inverse_estimate}\begin{split}
\int_\Omega \bigg[ h(v_{in}^\epsilon)^{p-1}\frac{(u_{A,in}^\epsilon)^p}{p}
+ k(v_{in}^\epsilon)^{p-1}\frac{(u_{B,in}^\epsilon)^p}{p} \bigg] \\
+ \,2\, d_A \frac{1-p}{p^2}\int_0^T\int_\Omega |\nabla_x [(u_A^\epsilon)^{p/2}] |^2 h(v^\epsilon)^{p-1}
+ 2\, (d_A+d_B) \frac{1-p}{p^2}\int_0^T\int_\Omega |\nabla_x [(u_B^\epsilon)^{p/2}] |^2 k(v^\epsilon)^{p-1}\\
-\frac{1}{\epsilon} \int_\Omega [k(v^\epsilon)u_{B}^\epsilon-h(v^\epsilon)u_{A}^\epsilon] [(u_B^\epsilon)^{p-1} k(v^\epsilon)^{p-1}-(u_A^\epsilon)^{p-1} h(v^\epsilon)^{p-1}]\\
\le
 \int_\Omega \bigg[ h(v^\epsilon)^{p-1}\frac{(u_A^\epsilon)^p}{p}(T)
+ k(v^\epsilon)^{p-1}\frac{(u_B^\epsilon)^p}{p}(T) \bigg] \\
+\frac{1-p}{p}\int_0^T\int_\Omega \partial_t v^\epsilon [h'(v^\epsilon)h(v^\epsilon)^{p-2}(u_A^\epsilon)^p+k'(v^\epsilon)k(v^\epsilon)^{p-2}(u_B^\epsilon)^p]\\
- \int_0^T\int_\Omega [r_u-r_{a}(u_{A}^\epsilon+u_{B}^\epsilon)^a-r_{b}(v^\epsilon)^b][(u_A^\epsilon)^{p} h(v^\epsilon)^{p-1}+(u_B^\epsilon)^{p} k(v^\epsilon)^{p-1}]\\
+ \frac{1-p}{2}\int_0^T \int_\Omega [d_A(u_A^\epsilon)^{p} (h'(v^\epsilon))^2 h(v^\epsilon)^{p-3}+(d_A+d_B)(u_B^\epsilon)^{p} (k'(v^\epsilon))^2 k(v^\epsilon)^{p-3}]|\nabla_x v^\epsilon|^2 .
\end{split} 
\end{equation}
Note that in estimate (\ref{entropy_inverse_estimate}), the first and third term of the r.h.s. are clearly
bounded (w.r.t. $\var \in ]0,1[$) thanks to estimates~(\ref{prelim_upr}), (\ref{prelim_vpr}),
 and (\ref{star}) (remember that 
$p \in ]0,1[$). 
\par
The second term is estimated thanks to the following inequality (remember that $p \in ]0, \inf(1, p^\ast - d) [$, and that estimates (\ref{star}), (\ref{starbi}) hold):

\begin{equation}\label{dt_v_in_entropy_estimate_pr}
\begin{split}
\left| \frac{1-p}{p}\int_0^T\int_\Omega \partial_t v^\epsilon [h'(v^\epsilon)h(v^\epsilon)^{p-2}(u_A^\epsilon)^p+k'(v^\epsilon)k(v^\epsilon)^{p-2}(u_B^\epsilon)^p]\right|\\
\le  \| h'(v^\epsilon)h(v^\epsilon)^{p-2}+k'(v^\epsilon)k(v^\epsilon)^{p-2}\|_{L^{\infty}}
\| \partial_t v^\epsilon \|_{L^{1+p/d}}
\| (u_A^\epsilon+u_B^\epsilon)^p \|_{L^{1+d/p}}
\le C_T\,(1+\|u_A^\epsilon+u_B^\epsilon\|_{L^{p+d}}^{p+d}) \le C_T.
\end{split} 
\end{equation}
Finally, the last term is estimated thanks to the inequality (we still use $p \in ]0, \inf(1, p^\ast - d) [$, and estimates (\ref{star}), (\ref{starbi})):

\begin{equation}\label{nabla_v_in_entropy_estimate_pr}
\begin{split}
\left| \frac{1-p}{2}\int_0^T \int_\Omega [d_A(u_A^\epsilon)^{p} (h'(v^\epsilon))^2 h(v^\epsilon)^{p-3}+(d_A+d_B)(u_B^\epsilon)^{p} (k'(v^\epsilon))^2 k(v^\epsilon)^{p-3}]|\nabla_x v^\epsilon|^2 \right|\\
\le (p/2)\, \| h'(v^\epsilon)^2h(v^\epsilon)^{p-3}+k'(v^\epsilon)^2k(v^\epsilon)^{p-3}\|_{L^{\infty}}
\| \, |\nabla_x v^\epsilon|^2 \|_{L^{1+p/d}}
\| (u_A^\epsilon+u_B^\epsilon)^p \|_{L^{1+d/p}}\\
\le C_T\,(1+\|u_A^\epsilon+u_B^\epsilon\|_{L^{p+d}}^{p+d}) \le C_T.
\end{split} 
\end{equation}
\medskip

Finally, we end up with the following (uniform w.r.t. $\var \in ]0,1[$) 
estimates (for $p \in ]0, \inf(1, p^\ast - d) [$):
\begin{equation}\label{estinj}
\int_0^T\int_\Omega |\nabla_x [(u_A^\epsilon)^{p/2}] |^2 h(v^\epsilon)^{p-1} \le C_T, \qquad
\int_0^T\int_\Omega |\nabla_x [(u_B^\epsilon)^{p/2}] |^2 k(v^\epsilon)^{p-1}\le C_T,
\end{equation}
and
\begin{equation}\label{estinj22}
- \frac{1}{\epsilon} \int_\Omega [k(v^\epsilon)u_{B}^\epsilon-h(v^\epsilon)u_{A}^\epsilon]
\, [(u_B^\epsilon)^{p-1} k(v^\epsilon)^{p-1}-(u_A^\epsilon)^{p-1} h(v^\epsilon)^{p-1}]
\le C_T.
\end{equation}

Remembering that $h,k$ lie in $C^1(\R_+)$, and that $v^\epsilon$ is uniformly bounded (thanks to
estimate (\ref{prelim_vpr})), we see that estimate (\ref{estinj}) implies
(for
$p \in ]0, \min(1, p^\ast -d)[$), the bound
\begin{equation}\label{estinj2}
||\nabla_x [(u_A^\var)^{p/2}] ||_{L^2} \le C_T, \qquad || \nabla_x [(u_B^\var)^{p/2}] ||_{L^2} \le C_T.
\end{equation}
Then, using the elementary inequality (for $p \in ]0,1[$)
$$ \forall x,y\in\R, \qquad  - (x-y)\, (x^{p-1} - y^{p-1}) \ge C_p \, | x^{p/2} - y^{p/2}|^2, $$
where $C_p>0$ is a constant (only depending on $p$),
we obtain (for
$p \in ]0, \min(1, p^\ast -d)[$),
$$ || (h(v^\var)\, u_A^\var)^{p/2} - (k(v^\var)\, u_B^\var)^{p/2} ||_{L^2}
 \le C_T\, \sqrt{\var}. $$
 
\smallskip

Moreover, thanks to estimate (\ref{star}), eq. (\ref{newse}) implies that
 $\pa_t(u_A^\var + u_B^\var)$ is bounded in $L^{\lambda}([0,T],W^{-2,\lambda})$ with $\lambda=\frac{p^\ast}{1+a}>1$ (remember that $a\le 1$).
 Finally, for $C=A,B$, we still can use the computation of estimate (\ref{twn}) and get, for 
 $p \in ]0, \min(1, p^\ast -d)[$, and selecting $\zeta=\zeta(p)\in ]0,1[$ such that 
$(2-p)\, \frac{1+\zeta}{1-\zeta} <1$, 
thanks to the bounds (\ref{estinj}) and (\ref{star}),
\begin{equation}\label{twn2}\begin{split}
\|\, |\nabla_x u_{C}^\epsilon|^{1 + \zeta} \, \|_{L^1} 
\le (2/p)^{1+\zeta} \, || \nabla_x (u_{C}^\epsilon)^{p/2}||_{L^2}^{1+\zeta}
\,\, 
\left( \int_0^T\int \left[ |u_{C}^\epsilon|^{(2-p)\,\frac{1+\zeta}{1-\zeta}}
\right] \right)^{\frac{1-\zeta}{2}} \le C_T.
\end{split}\end{equation}
We can therefore use Aubin's lemma and extract a subsequence
from $(u_A^\epsilon+u_B^\epsilon)_\epsilon$ (we keep the notation $(u_A^\epsilon+u_B^\epsilon)_\epsilon$ for this subsequence)
which converges towards a limit $u$ (lying in $L^{2}$, and nonnegative) for a.e. $(t,x) \in [0,T]\times\Omega$.
\medskip

Using the elementary inequality (\ref{forp}), inequality (\ref{starstar}) still holds  when $p \in ]0, \min(1, p^\ast -d)[$), and implies the convergences (\ref{convergence_hu_A-ku_B}),
(\ref{convergence_u_A}), (\ref{convergence_u_B}), (\ref{convul}) [with $a+p_0$ replaced by $p^*$].
Moreover, thanks to estimate (\ref{twn2}), the convergence (\ref{convergence_grad_u}) also holds.
\medskip

Then, as in Proposition~\ref{theo_sing1}, $u_A,u_B,v$ are
 defined on $\R_+\times\Omega$, and $u_A,u_B,v \ge 0$ a.e.
Moreover, $h(v)\, u_A = k(v)\,u_B$ a.e., and $u_A,u_B \in L^{p^\ast}$. Finally, we recall that
$v \in L^{\infty}$.
\medskip

Let us now show \eqref{es:Vp}. Thanks to the uniform (in $\var$) estimates \eqref{prelim_upr} and \eqref{estinj2}, we get for all $p\in]0, \min(1, p^\ast -d)[$,
\begin{equation}
\sup_{t\in [0,T]} \int_\Omega u (t) < +\infty \qquad \text{and} 
\qquad \int_0^T \int_\Omega \left( |\nabla_x u_A^{p/2}|^2+|\nabla_x u_B^{p/2}|^2 \right) <+\infty,
\end{equation}
where we have used Fatou's lemma for the first inequality and Kakutani's theorem applied to the reflexive space $L^2$ for the second inequality. 
We also recall that 
$\nabla_x v\in L^{2(1+p/d)}$ for all $p\in ]0, p^\ast -d]$. 
Using the identity, $u=\frac{h(v)+k(v)}{k(v)}\, u_A$, we see that for some $p>0$ small enough
\begin{align*}
\nabla_x \left(u^{p/2}\right) = \stackrel{\in L^{2 p^\ast /p}}{\overbrace{\phantom{\Bigg|} u_A^{p/2} }} \stackrel{\in L^\infty}{\overbrace{\left[\left(\frac{h+k}{k}\right)^{p/2}\right]'(v)}} \stackrel{\in L^{2(1+p/d)}}{\overbrace{\phantom{\Bigg|} \nabla_x v}} + \stackrel{\in L^\infty}{\overbrace{\left(\frac{h(v)+k(v)}{k(v)}\right)^{p/2}}} \,\stackrel{\in L^2}{\overbrace{\phantom{\Bigg|} \nabla_x \left(u_A^{p/2}\right)}} \;\in\; L^2,
\end{align*}
since 
(using $d\le2<p^\ast$) $\frac{1}{2p^\ast/p}+\frac{1}{2(1+p/d)}=\frac{p}{2p^\ast}+\frac{1}{2(1+p/d)} = \frac{1}{2}-\frac{p}{2}(\frac{1}{d}-\frac{1}{p^\ast}) + o_{p \to 0} (p) <\frac{1}{2}$ 
(remember that we take  $p>0$  small enough).
\medskip

We now briefly indicate how to pass to the limit in the various terms appearing in 
the approximate equations (\ref{wee1}) and (\ref{wee2}). Using estimate (\ref{star}), the 
uniform boundedness of $v^\var$ in $L^{\infty}$ and the weak convergence of $\nabla_x v^\epsilon$,
 we get
(\ref{weec1}) and (\ref{weec2}).
\par
 The same estimates imply that
$\psi_1\, (u_A^\epsilon+u_B^\epsilon)\, (r_u - r_{a}\, (u_A^\epsilon+u_B^\epsilon)^a - r_{b}\,(v^\epsilon)^b)$ is bounded in $L^{\frac{p^\ast}{1+a}}$, and
 $\psi_2\,
 v^\epsilon\, (r_v - r_{c}\,(v^\epsilon)^c - r_d\, (u_A^\epsilon+u_B^\epsilon)^d)$ 
 is bounded in $L^{\frac{p^\ast}{d}}$, so that we get (\ref{weec3}), (\ref{weec4}).
\par 
We know that $u_{A,in}^\epsilon + u_{B,in}^\epsilon \to u_{in}$ a.e. on $\Omega$. But
 $u_{in} \in L^{2}(\Omega)$, so that $u_{A,in}$ and $u_{B,in}$ also lie in $L^{2}(\Omega)$,
and $u_{A,in}^\epsilon,  u_{B,in}^\epsilon$ are bounded (uniformly w.r.t. $\var$) in 
$L^{2}(\Omega)$, so that we get (\ref{weec5}), (\ref{weec6}).
\par
Finally, the weak  convergence (in $L^1$) of $\nabla_x u_C^\epsilon$ towards $\nabla_x u_C$ (for $C=A,B$)
implies the convergence (\ref{weec7}), and the estimate $\nabla_x [\phi(v)\,u] \in L^1$.  
 \medskip

This concludes the proof of
Proposition~\ref{theo_sing2}.
\end{proof}

\section{Proof of existence, regularity and stability}\label{sec3}

In this section, we prove the Theorems \ref{theo_ex1} and \ref{theo_ex2}.
\bigskip

\begin{proof}[Proof of Theorem \ref{theo_ex1}] {\sl{First step: existence}}
\medskip

We use the notation $v_1:=\max \left(||v_{in}||_{L^{\infty}(\Omega)}, \left[\frac{r_v}{r_c\,(c+1)}\right]^{1/c}\right)$.
Thanks to a smooth cutoff function $\chi(v)$ ($\chi(v)=1$ for $0\le v\le v_1$, $\chi(v)=0$ for $v\ge 2 v_1$ and $0\le \chi(v) \le 1$ for all $v\ge0$), we define $\phi_B(v):=\chi(v)\,\phi(v)$ for all $v\ge0$.
 Since $\phi_B$ is a continuous function with compact support, it is bounded by some positive constant $\phi_1$.
\smallskip

Thanks to Assumption A satisfied by the parameters of  Theorem \ref{theo_ex1}, we see that
$d_u, d_v$, $r_u, r_v, r_a, r_b, r_c, r_d$, $a,b,c,d$ satisfy Assumption B of Proposition~\ref{theo_sing1}. Then we define 
 $d_A:=d_u/2$, $d_B:=d_u+\phi_1$, so that they also satisfy Assumption B (that is, they are 
strictly positive).  Finally we define the functions $h,k$ thanks to
$h(v):=d_u/2+\phi_B(v)$, 
 $k(v):=d_u/2+\phi_1-\phi_B(v)$. It is clear that $h,k \in C^1(\R_+)$ (because $\phi \in C^1(\R_+)$
 and $\chi$ is smooth). Moreover $h(v) \ge d_u/2>0$, $k(v) \ge d_u/2>0$, and
 $d_A+d_B\frac{h(v)}{h(v)+k(v)}=d_u+\phi_B(v)$. As a consequence, Assumption B is fulfilled
 except that $\phi(v)$ is replaced by $\phi_B(v)$.
 \smallskip
 
 Moreover, the extra assumptions on the parameters ($d<a$)  and on the initial data ($u_{in} \in L^{p_0}(\Omega)$, 
 $v_{in}\in L^\infty(\Omega)\cap W^{2, 1 + p_0/d}(\Omega)$ for some $p_0>1$) are the same 
 in Theorem \ref{theo_ex1} and Proposition~\ref{theo_sing1}.
\smallskip

 Then, 
 Proposition~\ref{theo_sing1} ensures that there exists a weak solution
 to system (\ref{sku1})--(\ref{sku4}) with $\phi(v)$ replaced by $\phi_B(v)$.
 Moreover, this solution $(u,v)$ has nonnegative components, $\nabla_x v \in L^{2(1+p_0/d)}_{\text{loc}}(\R_+\times\ov{\Omega})$, $u \in L^{p_0+a}_{\text{loc}}(\R_+\times\ov{\Omega})$,
 and for all $p\in]1,p_0]$, $T>0$,
\begin{equation}\label{reg_u}
 \sup_{t\in [0,T]} \int_\Omega u^{p_0} (t) < +\infty \qquad \text{and} \qquad \int_0^T \int_\Omega |\nabla_x u^{p/2}|^2<+\infty.
\end{equation}
Finally, $\nabla_x u, \nabla_x(u\,\phi(v)) \in L^{1}_{\text{loc}}(\R_+\times\ov{\Omega})$.
\medskip

 We also know that the bound $0\le v(t,x)\le v_1$ holds. By definition of $\phi_B$,
 we then have $\phi_B(v(t,x))=\phi(v(t,x))$ for all $t\ge0$, $x\in\Omega$, so that $(u,v)$ is in fact a weak solution of (\ref{sku1})--(\ref{sku4}), and this ends the proof of existence
in Theorem \ref{theo_ex1}.

\bigskip\noindent

{\sl{Second step: regularity, first part}} 
\medskip

We fix $T>0$ and define $p_1:=\max(2,\,a(s_0-1))$. By assumption, $u_{in}$ lies in $W^{2,s_0}(\Omega)$ with $s_0>1+N/2$, so that using a Sobolev embedding, $u_{in}$ lies in $L^{p_1}(\Omega)$. We also know (thanks to our assumptions) that $v_{in}\in W^{2,1+p_1/d}(\Omega)$. The results of the first step can therefore be 
obtained with $p_0$ replaced by $p_1$: in particular, estimate \eqref{reg_u} with $p_0$ replaced by $p_1$ implies that $u$ lies in $L^{p_1+a}([0,T] \times \Omega )$ and
\begin{equation}\label{u_V2}
\sup_{t\in [0,T]} \int_\Omega u^{2} (t) < +\infty \qquad \text{;} \qquad \int_0^T \int_\Omega |\nabla_x u|^2<+\infty.
\end{equation}

We now define $q_0:=(a+p_1)/d>s_0$. Using the maximal regularity for the (weak solutions of the) heat equation, we get (remember that $v$ lies in $L^\infty$)
\begin{equation}\label{prelim_v_CD}
\|\partial_t v\|_{L^{q_0}} \le C_T \,(1+\|u^d\|_{L^{q_0}})\le C_T, \hspace{1cm} \|\nabla_x^2 v\|_{L^{q_0}} \le C_T \,(1+\|u^d\|_{L^{q_0}})\le C_T.
\end{equation}

Using embedding results (see for example Lemma 3.3 in Chapter II of \cite{lsu}) and the fact that $q_0> 1 + N/2$, we see that $v$ is H\"older continuous on $[0,T] \times \ov{\Omega}$. 
\par
This shows that $v$ has the smoothness required in the theorem.
\smallskip

Similarly, $\partial_t \phi(v)= \phi'(v)\, \partial_t v$ and $\nabla_x^2 \phi(v)= \phi''(v)\,|\nabla_x v|^2  + \phi'(v) \nabla_x^2 v$ lie in $L^{q_0}$, so that $\phi(v) $ is also H\"older continuous 
on $[0,T] \times \ov{\Omega}$.
 We then rewrite the equation satisfied by $u$ as
\begin{equation}\label{eq:CD_u_parabolic}
\partial_{t} u - \nabla_x\cdot [A(t,x) \,\nabla_x u + B(t,x) u] +C(t,x)\,u =0,
\end{equation}
where $A =d_u +\phi(v)$ is H\"older continuous on $[0,T] \times \ov{\Omega}$,
$B = \nabla_x \phi(v)$ lies in $L^{2q_0}$,
and $C = - r_u+r_{a}u^a+r_{b}v^b$ lies in $L^{s_0}$. Note furthermore that $\nabla_x A=\nabla_x \phi(v)$ lies in $L^{2q_0}$ and $\nabla_x \cdot B=\Delta_x \phi(v) $ lies in $L^{s_0}$.
\medskip

We now recall two classical theorems from the theory of linear parabolic equations 
(see for example Theorem 5.1 in Chapter III of \cite{lsu} for the first one, and Theorem 9.1 and its corollary in Chapter IV of \cite{lsu} for the second one):
\begin{propo}\label{theo:parabolic_uniq}
Let $\Omega$ be a smooth bounded domain of $\R^N$ ($N\in \N^*$),  $T>0$ and $u_{in} \in L^{2}(\Omega)$. Consider the system
\begin{equation}\label{eq:linear_parabolic_w}\begin{split}
\partial_{t} u - \nabla_x\cdot [A(t,x) \,\nabla_x u + B(t,x) u] +C(t,x)\,u =0\quad {\hbox{ in }} [0,T] \times\Omega ,\\
\nabla_x u(t,x) \cdot n(x) = 0 \quad {\hbox{ on }} [0,T] \times\pa\Omega , \qquad u(0,\cdot) = u_{in} \quad {\hbox{ in }} \Omega ,
\end{split}\end{equation}
 where the coefficients satisfy: $A:= A(t,x)>0$ is continuous on $[0,T]\times \ov{\Omega}$, $B:=B(t,x)$ lies
 in $(L^{N+2})^N$, and $C:=C(t,x)$ lies in $L^{1+N/2}$.
\par
 A function $u:=u(t,x)$ is said to be a weak solution of \eqref{eq:linear_parabolic_w} (in the $V_2$ sense) if $u$ satisfies \eqref{u_V2} and, for all test functions $\psi \in C^1_c([0,T[ \times {\ov{\Omega}})$, 
the following identity holds:
\begin{equation*}
 - \int_0^{\infty}\int_{\Omega} (\pa_t \psi)\, u - \int_{\Omega} \psi(0,\cdot)\, u_{in} 
+ \int_0^{\infty}\int_{\Omega} [A \, \nabla_x u + B\, u ] \cdot \nabla_x \psi +  \int_0^{\infty}\int_{\Omega}   C\,u\,\psi =0. 
\end{equation*}
Notice that all terms in the previous identity are well defined when $u,\psi$, $A,B,C$ satisfy the assumptions
of Proposition \ref{theo:parabolic_uniq} (cf. estimate (3.4) in Chapter II of \cite{lsu}).
\par
 Then system \eqref{eq:linear_parabolic_w} has at most one weak solution (in the $V_2$ sense).
\end{propo}

\begin{propo}\label{theo:parabolic_reg}
Let $\Omega$ be a smooth bounded domain of $\R^N$ ($N\in \N^*$), 
 $s>1 + N/2$ and $T>0$. Consider the system
\begin{equation}\label{eq:linear_parabolic}\begin{split}
\partial_{t} u - A(t,x) \,\Delta_x u + B_1(t,x) \cdot \nabla_x u +C_1(t,x)\,u =0\quad {\hbox{ in }} [0,T] \times\Omega ,\\
\nabla_x u(t,x) \cdot n(x) = 0 \quad {\hbox{ on }} [0,T] \times\pa\Omega , \qquad u(0,\cdot) = u_{in} \quad {\hbox{ in }} \Omega ,
\end{split}\end{equation}
 where the coefficients satisfy: $A:= A(t,x)>0$ is continuous on $[0,T]\times \ov{\Omega}$, $B_1:=B_1(t,x)$ lies
 in $(L^r)^N$ for some $r>\max(s,N+2)$, and $C_1:=C_1(t,x)$ lies in $L^s$.
 Suppose also that $u_{in} \in W^{2,s}(\Omega)$ (and, if $s\ge3$, that the compatibility condition $\nabla_x u_{in}(x) \cdot n(x) = 0$ on
 $\partial\Omega$ holds). 
\par
A function $u:=u(t,x)$ is said to be a strong solution of \eqref{eq:linear_parabolic} (in the $W^{1,2}_s$ sense) if $\partial_t u$ and $\pa^2_{x_ix_j} u$ lie in $L^s$ (for $i,j=1..N$) and system \eqref{eq:linear_parabolic} is satisfied almost everywhere in $[0,T] \times\Omega$ (resp. $[0,T] \times\pa\Omega$, resp. $\Omega$).
\par
 Then, system \eqref{eq:linear_parabolic} has a unique strong solution $u$ (in the $W^{1,2}_s$ sense). Furthermore, $u$ is H\"older continuous on $[0,T] \times \ov{\Omega}$.
\end{propo}
A direct consequence of these two propositions is given by the
\begin{corol}\label{theo:weak_parabolic_reg}
Let $\Omega$ be a smooth bounded domain of $\R^N$ ($N\in \N^*$),
 $s>1 + N/2$ and $T>0$. We assume that $u_{in}$, $A$, $B$, $C$, $B_1:=-B-\nabla_x A$ and $C_1:=C-\nabla_x \cdot B$ satisfy the requirements of Propositions \ref{theo:parabolic_uniq} and \ref{theo:parabolic_reg}.
\par
Then any weak solution $u$ of system \eqref{eq:linear_parabolic_w}, or equivalently system \eqref{eq:linear_parabolic}, (in the $V_2$ sense) is a strong solution (in the $W^{1,2}_s$ sense). In particular, $\partial_t u$ and $\pa^2_{x_ix_j} u$ lie in $L^s$ (for $i,j=1..N$). Furthermore, $u$ is H\"older continuous on $[0,T] \times \ov{\Omega}$.
\end{corol}

We now come back to the second step of the proof of Theorem \ref{theo_ex1}.
Using Corollary \ref{theo:weak_parabolic_reg} with $s=s_0$, we see that $u$ has the smoothness required in the theorem.
This concludes the second step of the proof of Theorem \ref{theo_ex1}, that is the first part of the study of regularity.
\bigskip

{\sl{Third step: regularity, second part}}
\medskip

We now assume that
$\phi$, (resp. $u_{in}, v_{in}$) have H\"older continuous second order derivatives on $\R_+$ (resp. $\ov{\Omega}$). We fix $T>0$.
\par
We already know that $u$ and $v$ are H\"older continuous on $[0,T] \times \ov{\Omega}$. 
 It is then clear that in eq. (\ref{sku2}), the reaction term is H\"older continuous on $[0,T] \times \ov{\Omega}$.
 Thanks to standard results in the theory of linear parabolic equations (see for example Theorem 5.3 in Chapter IV of \cite{lsu}), $\partial_t v$ and $\nabla_x^2 v$ are also H\"older continuous on $[0,T] \times \ov{\Omega}$. 
Writing eq. (\ref{sku1}) in its form \eqref{eq:linear_parabolic}, we see that the coefficients $A$, $B_1:=-B-\nabla_x A$ and $C_1:=C-\nabla_x \cdot B$ are H\"older continuous on  $[0,T] \times \ov{\Omega}$ (note that we use here 
the H\"older continuity of $\phi''$). The same result for linear parabolic equations implies  that $\partial_t u$ and $\nabla_x^2 u$ are H\"older continuous on $[0,T] \times \ov{\Omega}$.
\medskip

This concludes the second step of the study of the regularity.
\bigskip

{\sl{Fourth step: stability and uniqueness}}
\medskip

We still assume that $\phi$, (resp. $u_{in}, v_{in}$) have H\"older continuous second order derivatives on $\R_+$ (resp. $\ov{\Omega}$).
\par
Let $(u_1,v_1)$ and $(u_2,v_2)$ be two weak solutions of \eqref{sku1}-\eqref{sku4}
in the sense of Definition \ref{defimer} satisfying the assumptions of the theorem. Recall the definition of $p_1$ in the second step, and notice that by assumption $u_1$, $u_2\in L^{p_1+a}$. Moreover, 
 estimate \eqref{u_V2} with $u=u_1,\,u_2$ holds. Therefore the computations of the second and third steps are valid for $(u,v)=(u_1,v_1),\, (u_2,v_2)$. This implies that these solutions $(u_1,v_1)$ and $(u_2,v_2)$ are continuous (and even H\"older continuous) functions on $[0,T] \times \ov{\Omega}$, and so are the space gradients $\nabla_x v_1$ and $\nabla_x v_2$. 
\medskip

For any function $(u,v)\mapsto F(u,v)$, we write $\overline{F(u,v)}=\frac{F(u_1,v_1)+F(u_2,v_2)}{2}$. 
\medskip

We substract the equations satisfied by $(u_2,v_2)$ to the equations satisfied by $(u_1,v_1)$, 
and get

\begin{equation} \label{eq:cross_diff_difference} \begin{aligned}
	\partial_{t} (u_1-u_2)& - \Delta_x [(d_A +\overline{\phi(v)})\,(u_1-u_2)]- \Delta_x [(\phi(v_1)-\phi(v_2))\,\overline{u}]\\
	&= [r_v - r_{a}\,\overline{u^a} - r_{b}\,\overline{v^b}]\, (u_1-u_2) - [r_{a}\,(u_1^a-u_2^a)+r_{b}\,(v_1^b-v_2^b)]\, \overline{u},\\
	\partial_{t} (v_1-v_2)& - d_v\,\Delta_x (v_1-v_2)\\
	&= [r_v-r_{c}\,\overline{v^c}-r_{d}\,\overline{u^d}]\,(v_1-v_2) - [r_{c}\,(v_1^c-v_2^c)+r_{d}\,(u_1^d-u_2^d)]\,\overline{v}.
\end{aligned} \end{equation}

We multiply the first equation by the difference $u_1-u_2$ and integrate w.r.t. space and time. We get the identity

\begin{equation} \label{eq:L2_energy_difference_u} \begin{aligned}
	\frac{1}{2} \int_\Omega (u_1-u_2)^2(T)& 
	+ \int_0^T\int_\Omega (d_A +\overline{\phi(v)})\,|\nabla_x(u_1-u_2)|^2
	+\int_0^T\int_\Omega (u_1-u_2)\nabla_x(u_1-u_2)\cdot \nabla_x (\overline{\phi(v)})\\
	&+\int_0^T\int_\Omega (\phi(v_1)-\phi(v_2))\nabla_x(u_1-u_2)\cdot \nabla_x \overline{u}
	+\int_0^T\int_\Omega \overline{u}\,\nabla_x(u_1-u_2)\cdot\nabla_x [\phi(v_1)-\phi(v_2)]\\
	= \frac{1}{2} \int_\Omega (u_1-u_2)^2(0) &+ \int_0^T\int_\Omega [r_v - r_{a}\overline{u^a} - r_{b}\overline{v^b}]\, (u_1-u_2)^2 - \int_0^T\int_\Omega (u_1-u_2)\,[r_{a}\,(u_1^a-u_2^a)+r_{b}\,(v_1^b-v_2^b)]\, \overline{u}.
\end{aligned} \end{equation}

In the left-hand side of this identity, the two first terms are nonnegative.
 The other terms are controlled thanks to the smoothness of the functions $(\overline{u},\overline{v})$ and their space gradients (and the elementary inequality $2ab\le a^2+b^2$).
We detail below their treatment: the third term of \eqref{eq:L2_energy_difference_u} is controlled by
\begin{equation}\begin{split}
\left|\int_0^T\int_\Omega (u_1-u_2)\,\nabla_x(u_1-u_2)\cdot \nabla_x (\overline{\phi(v)})\right|
&\le C_T \int_0^T\int_\Omega |u_1-u_2|\, |\nabla_x(u_1-u_2)|\\
&\le \frac{d_A}{4} \int_0^T\int_\Omega |\nabla_x(u_1-u_2)|^2 + C_T \int_0^T\int_\Omega |u_1-u_2|^2,
\end{split}\end{equation}
the fourth term of \eqref{eq:L2_energy_difference_u} is controlled by
\begin{equation}\begin{split}
\left|\int_0^T\int_\Omega (\phi(v_1)-\phi(v_2))\,\nabla_x(u_1-u_2)\cdot \nabla_x \overline{u}\right|
&\le \frac{d_A}{4} \int_0^T\int_\Omega |\nabla_x(u_1-u_2)|^2 + C_T \int_0^T\int_\Omega |\phi(v_1)-\phi(v_2)|^2,
\end{split}\end{equation}
and the fifth term of \eqref{eq:L2_energy_difference_u} is controlled by
\begin{equation}\begin{split}
\left|\int_0^T\int_\Omega \overline{u}\,\nabla_x(u_1-u_2)\cdot\nabla_x 
[\phi(v_1)-\phi(v_2)]\right|
&\le \frac{d_A}{4} \int_0^T\int_\Omega |\nabla_x(u_1-u_2)|^2 + C_T \int_0^T\int_\Omega |\nabla_x[\phi(v_1)-\phi(v_2)]|^2,
\end{split}\end{equation}
where moreover
\begin{equation}\begin{split}
\int_0^T\int_\Omega |\nabla_x[\phi(v_1)-\phi(v_2)]|^2 = \int_0^T\int_\Omega |\overline{\phi'(v)}\,\nabla_x (v_1 - v_2) + (\phi'(v_1)-\phi'(v_2))\,\nabla_x \overline{v} |^2\\
\le C_T \int_0^T\int_\Omega |\nabla_x (v_1 - v_2)|^2
+ C_T \int_0^T\int_\Omega |\phi'(v_1)-\phi'(v_2) |^2.
\end{split}\end{equation}
It remains to control the last term of the right-hand side :
\begin{equation}\begin{split}
- \int_0^T\int_\Omega (u_1-u_2)\,[r_{a}\,(u_1^a-u_2^a)+r_{b}\,(v_1^b-v_2^b)]\, \overline{u}
\le r_{b} \int_0^T\int_\Omega |u_1-u_2| \,|v_1^b-v_2^b|\, \overline{u}\\
\le C_T \int_0^T\int_\Omega |u_1-u_2|^2 + C_T \int_0^T\int_\Omega |v_1^b-v_2^b|^2.
\end{split}\end{equation}
Thanks to those estimates, the identity \eqref{eq:L2_energy_difference_u} becomes
\begin{equation}\label{eq:L2_energy_difference_u_est} \begin{aligned}
	\int_\Omega (u_1-u_2)^2(T)
	\le & \int_\Omega (u_1-u_2)^2(0) + C_T \,\left( \int_0^T\int_\Omega (u_1-u_2)^2 + \int_0^T\int_\Omega |\phi(v_1)- \phi(v_2)|^2\right.\\ &
+  \int_0^T\int_\Omega |\phi'(v_1)-\phi'(v_2) |^2 
 + \left.\int_0^T\int_\Omega |\nabla_x (v_1 - v_2)|^2 + \int_0^T\int_\Omega |v_1^b-v_2^b|^2\right).
\end{aligned} \end{equation}
\smallskip

We now multiply the second equation of \eqref{eq:cross_diff_difference} by the difference $v_1-v_2$ and integrate w.r.t.
 space and time. We get
\begin{equation} \label{eq:L2_energy_difference_v} \begin{aligned}	
	&\frac{1}{2}\int_\Omega (v_1-v_2)^2(T) + d_v \int_0^T\int_\Omega |\nabla_x (v_1-v_2)|^2\\
    &=\frac{1}{2}\int_\Omega (v_1-v_2)^2(0)	\\
&+ \int_0^T\int_\Omega [r_v-r_{c}\,\overline{v^c}-r_{d}\,\overline{u^d}]\,(v_1-v_2)^2 - \int_0^T\int_\Omega (v_1-v_2)\,[r_{c}\,(v_1^c-v_2^c)+r_{d}\,(u_1^d-u_2^d)]\,\overline{v}\\
	&\le \frac{1}{2}\int_\Omega (v_1-v_2)^2(0) + C_T \bigg( \int_0^T\int_\Omega (v_1-v_2)^2 + 
\int_0^T\int_\Omega |u_1^d-u_2^d|^2
\bigg) .
\end{aligned} \end{equation}

We combine the two energy estimates \eqref{eq:L2_energy_difference_u_est} and \eqref{eq:L2_energy_difference_v}:
\begin{equation}\label{eq:L2_energy_difference_uv} \begin{aligned}
	&\int_\Omega (u_1-u_2)^2(T)+\int_\Omega (v_1-v_2)^2(T)
	\le \int_\Omega (u_1-u_2)^2(0)+\int_\Omega (v_1-v_2)^2(0) \\
 + \,  &C_T \left( \int_0^T\int_\Omega (u_1-u_2)^2 + \int_0^T\int_\Omega (v_1-v_2)^2 
\right.\\
 + \int_0^T\int_\Omega |u_1^d-u_2^d|^2 &+ \int_0^T\int_\Omega |\phi(v_1)- \phi(v_2)|^+ \left. \int_0^T\int_\Omega |v_1^b-v_2^b|^2 + \int_0^T\int_\Omega |\phi'(v_1)-\phi'(v_2)|^2 \right).
\end{aligned} \end{equation}
Since $\phi''$ is continuous on $\mathbb{R}_+$, the applications $\phi$ and $\phi'$ are locally Lipschitz on $\mathbb{R}_+$. 
The assumption $b\ge 1, d\ge 1$ ensures that the applications $v\mapsto v^b$ and $u\mapsto u^d$ are also locally Lipschitz on $\mathbb{R}_+$. Therefore
\begin{equation}
	\int_\Omega (u_1-u_2)^2(T)+\int_\Omega (v_1-v_2)^2(T)
	\le \int_\Omega (u_1-u_2)^2(0)+\int_\Omega (v_1-v_2)^2(0) 
\end{equation}
$$ + \,C_T\, \left( \int_0^T\int_\Omega (u_1-u_2)^2 + \int_0^T\int_\Omega (v_1-v_2)^2 \right) ,$$
and we can conclude thanks to Gronwall's lemma.
\medskip

Note that thanks to the minimum principle, the assumption $b\ge 1, d\ge 1$ can be relaxed if the initial data $u_{in}$ and $v_{in}$
are bounded below by a strictly positive constant.
\medskip

This concludes the study of stability (and uniqueness), and ends the proof of Theorem \ref{theo_ex1}.

\end{proof}

\bigskip

\begin{proof}[Proof of Theorem \ref{theo_ex2}]

As in the proof of Theorem \ref{theo_ex1}, we use the notation $v_1:=\max \left(||v_{in}||_{L^{\infty}(\Omega)}, \left[\frac{r_v}{r_c\,(c+1)}\right]^{1/c} \right)$. We also
introduce a smooth cutoff function $\chi(v)$ ($\chi(v)=1$ for $0\le v\le v_1$, $\chi(v)=0$ for $v\ge 2 v_1$ and $0\le \chi(v) \le 1$ for all $v\ge0$), together with $\phi_B(v):=\chi(v)\,\phi(v)$ (for all $v\ge0$), and an upper bound $\phi_1$ for $\phi_B$.
\smallskip

Thanks to Assumption A satisfied by the parameters of  Theorem \ref{theo_ex2}, we see that
$d_u, d_v$, $r_u, r_v, r_a, r_b, r_c, r_d$, $a,b,c,d$ satisfy Assumption B of Proposition~\ref{theo_sing2}. Then we define, as in the proof of Theorem \ref{theo_ex1},
 $d_A:=d_u/2$, $d_B:=d_u+\phi_1$, so that they  satisfy Assumption B,
and the functions $h,k$ thanks to
$h(v):=d_u/2+\phi_B(v)$, 
 $k(v):=d_u/2+\phi_1-\phi_B(v)$. It is clear that $h,k \in C^1(\R_+)$
and  $h(v) \ge d_u/2>0$, $k(v) \ge d_u/2>0$, and
 $d_A+d_B\frac{h(v)}{h(v)+k(v)}=d_u+\phi_B(v)$. As a consequence, Assumption B is fulfilled
 except that $\phi(v)$ is replaced by $\phi_B(v)$.
 \smallskip
 
 Moreover, the extra assumptions on the parameters ($d \ge a$, $a\le 1$, $d\le 2$)  and on the initial data ($u_{in} \in L^{2}(\Omega)$, 
 $v_{in}\in L^\infty(\Omega)\cap W^{2, 1 + 2/d}(\Omega)$) are the same 
 in Theorem \ref{theo_ex2} and Proposition~\ref{theo_sing2}.
\smallskip

 Then, Proposition~\ref{theo_sing2} ensures that there exists a weak solution
 to system (\ref{sku1})--(\ref{sku4}) with $\phi(v)$ replaced by $\phi_B(v)$.
 Moreover, this solution $(u,v)$ has nonnegative components and lies in 
 $L^{2}_{\text{loc}}(\R_+\times\ov{\Omega})\times L^\infty_{\text{loc}}(\R_+\times\ov{\Omega})$.
We also know that 
 for some $p>0$,  
$u$ satisfies \eqref{es:Vp}. Moreover, we know that $\nabla_x v \in L^{2 + \eta}_{\text{loc}}(\R_+\times\ov{\Omega})$, $\nabla_x u , \nabla_x (u\,\phi(v))\in L^{1}_{\text{loc}}(\R_+\times\ov{\Omega})$, for some $\eta>0$.
\par
Finally, we know that the bound $0\le v(t,x)\le v_1$ holds, so that
 by definition of $\phi_B$,
 we see that $\phi_B(v(t,x))=\phi(v(t,x))$ for all $t\ge0$, $x\in\Omega$. Then,  $(u,v)$ is in fact a weak solution of (\ref{sku1})--(\ref{sku4}).
This ends the proof of Theorem~\ref{theo_ex2}.
\end{proof}

\end{document}